\documentclass[10pt,a4paper]{article}
\usepackage[utf8]{inputenc}
\usepackage{amsmath}
\usepackage{amsfonts}
\usepackage{amssymb}
\usepackage{mathtools}
\usepackage{xcolor}

\usepackage{amsthm}
\usepackage{bbm}
\usepackage{graphicx}
\usepackage{float}

\newtheorem{thm}{Theorem}[section]
\newtheorem{lemma}[thm]{Lemma}
\newtheorem{proposition}[thm]{Proposition}
\newtheorem{corollary}[thm]{Corollary}
\newtheorem{conjecture}[thm]{Conjecture}
\theoremstyle{definition}
\newtheorem{example}[thm]{Example}
\newtheorem{definition}[thm]{Definition}
\newtheorem{remark}[thm]{Remark}

\newtheorem{theorem}{Theorem}

\newcommand{\R}{{\mathbb{R}}}

\newcommand{\N}{{\mathbb{N}}}

\newcommand{\dg}{\dot{\gamma}}

\newcommand{\ehzcap}{c_{_{\rm EHZ}}}

\newcommand{\deta}{\dot{\eta}}
\newcommand{\dzeta}{\dot{\zeta}}

\title{Dynamical extensions of Zoll to nonsmooth convex bodies}
\author{Pazit Haim-Kislev }
\date{2025}

\begin{document}

\maketitle

\begin{abstract}
The symplectic Zoll property of smooth convex domains in the classical phase space has been extensively studied in recent years and, in particular, has been shown to detect local maximizers of the systolic ratio.
We propose a dynamical extension of this property to the non-smooth setting, related to the behavior of the Ekeland–Hofer–Zehnder capacity with respect to hyperplane cuts.
Under certain hypotheses, we establish its equivalence to a known topological extension of the Zoll property. 

In this context, we study the space of non-smooth action-minimizing closed characteristics and show that their dynamical behavior is not as irregular as one might first expect.
We classify several types of dynamical behaviors and derive a certain $H^1$-compactness result, which is of independent interest.

\end{abstract}

\section{Introduction and results}

Let $\R^{2n}$ be the classical phase space equipped with the standard symplectic structure 
\[
\omega \;=\; \sum_{i=1}^n dp_i \wedge dq_i .
\]
For a smooth star-shaped body $K \subset \R^{2n}$, the \emph{characteristic line bundle} of $\partial K$ is the kernel
\[
\ker\bigl(\omega|_{T\partial K}\bigr) \subset T\partial K .
\]
 A \emph{closed charactersitc} on $\partial K$ is a smooth map $\gamma : S^1 \to \partial K$ whose tangent vector is everywhere contained in this line bundle.
 Equivalently, closed characteristics (up to reparametrization) are closed Reeb orbits on the contact manifold $\partial K$ with respect to the contact form \[ \frac{1}{2} \sum_{i=1}^n p_i dq_i - q_i dp_i.\]
Closed characteristics on the boundaries of star-shaped domains in $\R^{2n}$ lie at the heart of symplectic dynamics (see e.g. \cite{ABHS, dynconvnotconvdim4, HWZ-3dconvex, bramhamHofer, viterboWeinsteinConj}).

In this paper, we are interested in the case where $K$ is a (bounded) convex body. 
By works of Weinstein \cite{weinstein78}, Rabinowitz \cite{rabinowitz78}, and Clarke \cite{clarke}, every compact convex hypersurface carries a closed Reeb orbit. 
Define the Ekeland--Hofer--Zehnder (EHZ) capacity $\ehzcap(K)$ to be the minimal action of a closed characteristic on $\partial K$. Note that the EHZ capacity does not satisfy the properties of a symplecitc capacity in general, when the domains are not assumed to be convex (we refer to, e.g., \cite{capacity_survey_1, capacity_survey_2, capacity_survey_3} for information on symplectic capacities).
An application of Clarke's dual action principle \cite{clarke} implies monotonicity with respect to inclusions: for convex domains $K_1 \subseteq K_2$ one has $\ehzcap(K_1) \leq \ehzcap(K_2)$. More generally, even if the image of a symplectic embedding of $K_1$ into $K_2$ is not convex, it still holds that $\ehzcap(K_1) \leq \ehzcap(K_2)$. However, this monotonicity with respect to (not necessarily convex) symplectic embeddings is extremely non-trivial, and in particular it recovers the famous Gromov's non-squeezing result \cite{gromov}. It is proved by the coincidence of the EHZ capacity with the restriction of genuine symplectic capacities, such as the Hofer--Zehnder capacity \cite{hofer_zehnder}, to the class of convex bodies.

Viewing the EHZ capacity as a ``symplectic size" of convex sets, it is natural to compare it with the volume of these sets.
A central quantitative theme is thus the \emph{systolic ratio}
$$ \rho_{\text{sys}}(K) := \frac{\ehzcap(K)^n}{n! \text{Vol}(K)}.$$
A natural follow-up question is to understand the maximizers of the systolic ratio, and this question has attracted much attention over the years (a partial list of relevant references includes \cite{abbondandoloBenedetti,  abbondandolo-benedetti-edtmair,ABHS, ArtsteinAvidanOstroverMilman, balitsky,Ch-H,edtmair, gutt-hutchings-ramos,hermann, karasev-sharipova, viterbo2000, rudolf}).
In particular, Viterbo \cite{viterbo2000} conjectured that the systolic ratio is maximized by the ball. This conjecture has influenced many subsequent works and revealed connections to a variety of other mathematical areas (see e.g. \cite{capacity_mahler}). In a recent joint work with Ostrover, we introduced a counterexample to Viterbo's conjecture \cite{counterexample}, thereby leaving us, at present, without a clear candidate for a maximizer.
We remark that there is still substantial evidence for the existence of an underlying principle explaining why the conjecture holds in many special cases, and it remains an intriguing problem to identify the appropriate subclass of convex domains for which Viterbo's conjecture is valid.

Following \cite{ABHS, abbondandoloBenedetti, abbondandolo-benedetti-edtmair}, it is known that smooth convex domains that satisfy the Zoll property are local maximizers of $\rho_{\text{sys}}$. The term Zoll is borrowed from Riemannian geometry and it means in the symplectic setting that $\partial K$ is foliated by action-minimizing closed characteristics. 
In addition, strictly convex local maximizers are Zoll.
It is known that Zoll convex domains satisfy Viterbo's conjecture (cf. \cite{boothbyWang} and \cite[Section 7.2]{Geiges_2008}). Moreover, in dimension four all smooth Zoll domains are symplectic balls \cite{ABHS}.

The space of smooth convex domains is, however, not compact. Building on the fact that the EHZ capacity is continuous with respect to the Hausdorff topology, there is a unique extension to non-smooth convex domains. 
This extension also has a dynamic interpretation. It can be defined directly, without a limiting procedure, using actions of \emph{generalized} characteristics defined via normal cones in the sense of non-smooth convex analysis (see Section~\ref{nonSmoothSection}).
Using John ellipsoids (see e.g. \cite[Section 10.12]{subdiff_1}) and compactness of (non-smooth) convex domains, there exists a global maximizer of $\rho_{\text{sys}}$. In particular, in view of the counterexample in \cite{counterexample}, there are smooth convex domains that do not satisfy Viterbo's conjecture which converge to a non-smooth maximizer. 
This makes understanding the dynamics on non-smooth convex domains particularly interesting.

A natural question is therefore to classify different types of dynamical behaviors of characteristics on the boundaries of non-smooth convex domains, and in particular recognize local maximizers of the systolic ratio. In light of the above discussion, one guiding question is how to extend the notion of a Zoll domain to the non-smooth setting.
As opposed to the smooth case, the property of being foliated by action-minimizing characteristics does not imply local maximality of $\rho_{\text{sys}}.$ For instance, the cube $Q = [0,1]^4$ is foliated by minimizers, yet $Q$ is far from being a local maximizer as one can decrease its volume by half while maintaining the same EHZ capacity.

In this paper, we offer a dynamical property which offers a natural extension of the Zoll property to non-smooth domains. We show that it is equivalent to the Zoll property in the smooth strictly convex case, and derive equivalences to other topological extensions of Zoll under certain hypotheses.
We then introduce new results regarding dynamical properties of action-minimizing closed characteristics on boundaries of non-smooth convex domains. 

\subsection{Dynamical and topological extensions of the Zoll property}

One approach to extending the Zoll property in the non-smooth setting uses the Ekeland--Hofer (EH) capacities \cite{hofer-ekeland} denoted $c_k(K)$ (or the Gutt--Hutchings capacities \cite{Gu-Ha} which were recently shown to coincide with the EH capacities \cite{GuttRamos}). Ginzburg, G\"{u}rel, and Mazzucchelli \cite{Gi-Gu-Ma} showed that for $K\subset \R^{2n}$ smooth and convex, $K$ is Zoll if and only if $c_1(K) = c_n(K)$. They also showed that this happens if and only if the Fadell-Rabinowitz (FR) index of the $S^1$-space of action-minimizing closed characteristics is $n$. Following Matijevi\'{c} \cite{Matijevic}, we call a not necessarily smooth convex $K \subset \R^{2n}$ generalized Zoll if $c_1(K) = c_n(K),$ which, as was shown ibid, is equivalent to the requirement that the FR index of action-minimizers is greater than or equal to $n$.
However, it remains unclear whether, for a general non-smooth domain, the generalized Zoll property implies that $\partial K$ is foliated by action minimizers.

\medskip

In this paper, we propose a second, more dynamical approach, which involves the behavior of $\ehzcap(K)$ under hyperplane cuts.
Following \cite{pazit}, the EHZ capacity is subadditive with respect to hyperplane cuts: if a hyperplane $H$ cuts $K$ into $K_1$ and $K_2$ (see Figure~\ref{additivityFig}) then 
$$\ehzcap(K) \leq \ehzcap(K_1) + \ehzcap(K_2).$$
On the other hand, for symplectic images of the ball, a result of Zehmish \cite{ball_add} implies \emph{superadditivity} (see also \cite{cant} for a different proof of this fact).
Motivated by this, we make the following definition.
\begin{definition}
We say that a convex body $K$ is \emph{cuts additive} if for every hyperplane $H$ cutting $K$ into $K_1$ and $K_2$ one has
\[
\ehzcap(K) = \ehzcap(K_1) + \ehzcap(K_2).
\]
\end{definition}

For a convex domain $K$, we say that the \emph{characteristic dynamics is well defined} if at every point $x \in \partial K$ passes only one characteristic up to time reparametrizations. This condition is satisfied if $K$ is smooth.
Our first result deals with this situation.

\begin{theorem}[Generalized Zoll implies Cuts Additive] \label{ZollIsAdditiveWellDefinedThm}
Let $K \subset \R^{2n}$ be a generalized Zoll convex domain for which the characteristic dynamics is well defined. Then $K$ is cuts additive. 
\end{theorem}
\noindent We remark that in \cite{Matijevic} it is shown that if the characteristic dynamics is well defined then the generalized Zoll property is equivalent to requiring that $\partial K$ is foliated by action minimizing closed characteristics.

\medskip

The equivalence between Zoll and cuts additivity holds also in the other direction when restricting to the strictly convex case.
In this case, it is enough to test for cuts additivity only in a small neighborhood of the boundary (see the proof of Corollary~\ref{extremizerIsAdditiveCor}).

For describing it let us first give the following definition.
\begin{definition}
We say that a convex body $K$ is a \emph{cuts extremizer} if every hyperplane cut strictly decreases the EHZ capacity, i.e.\ for every such cut $K=K_1\cup K_2$ one has
\[
\ehzcap(K_1),\ehzcap(K_2) < \ehzcap(K).
\]
\end{definition}

\begin{corollary} \label{extremizerIsAdditiveCor}
    If $K$ is smooth, strictly convex, and cuts extremizer then $K$ is Zoll and cuts additive.
\end{corollary}
\begin{proof}
    In view of Theorem~\ref{ZollIsAdditiveWellDefinedThm} it is enough to show that $K$ is Zoll.
    
    Let $x \in \partial K$, and let $\gamma_x$ be the characteristic curve emanating from $x$. 
    If $\gamma_x$ is not an action minimizer closed characteristic, strict convexity implies that there is an open neighborhood $U \subset \partial K$ of $\gamma_x(t)$ which contains no action-minimizing closed characteristics.
    A sufficiently small hyperplane cut supported inside $U$ does not reduce the capacity, contradicting the cuts extremizer property.

\end{proof}

\noindent Example~\ref{pentagonEx} below shows that, in general, being cuts extremizer does not imply cuts additivity.

A consequence of Theorem~\ref{ZollIsAdditiveWellDefinedThm} and Corollary~\ref{extremizerIsAdditiveCor} is that cuts additivity is preserved under symplectomorphisms that keep $K$ smooth and strictly convex, although this does not follow directly from the definition.

Theorem~\ref{ZollIsAdditiveWellDefinedThm} suggests that cuts additivity serves as an alternative natural extension of Zoll for non-smooth domains.
Roughly speaking, the FR index of action minimizers can be regarded as a topological extension of Zoll to non-smooth domains, while cuts additivity provides a dynamical one. 
Indeed, cuts additivity admits a concrete dynamical interpretation (see Figure~\ref{additivityFig}). 
\begin{proposition} \label{additiveDescriptionThm}
Let $K \subset \R^{2n}$ be convex, and let $H$ be a hyperplane with unit normal $u$ which cuts $K$ into $K_1$ and $K_2$. Then cuts additivity with respect to this cut is equivalent to the existence of an action-minimizing closed characteristic $\gamma : S^1 \to \partial K$ that intersects $H$ twice at times $t_1<t_2$ and satisfies
    \[
        \gamma(t_2) - \gamma(t_1) \in \R_+\, J u.
    \]
    In this case, the restriction of $\gamma$ to $(t_1,t_2)$ glued with the straight segment $[\gamma(t_2),\gamma(t_1)]$ is an action-minimizing closed characteristic on $\partial K_1$, and the restriction of $\gamma$ to the complement of $(t_1,t_2)$ glued with the segment $[\gamma(t_1),\gamma(t_2)]$ is an action-minimizer for $\partial K_2$.
    
    Moreover, every pair of action-minimizing characteristics $\gamma_1,\gamma_2$ on $\partial K_1, \partial K_2$, up to translations, can be glued together to form an action-minimizing closed characteristic on $\partial K$.  
\end{proposition}

\noindent This proposition can be viewed as studying the equality case in the proof of \cite[Theorem 1.8]{pazit}.

\begin{figure}[H]
    \centering
    \includegraphics[width=0.7\linewidth]{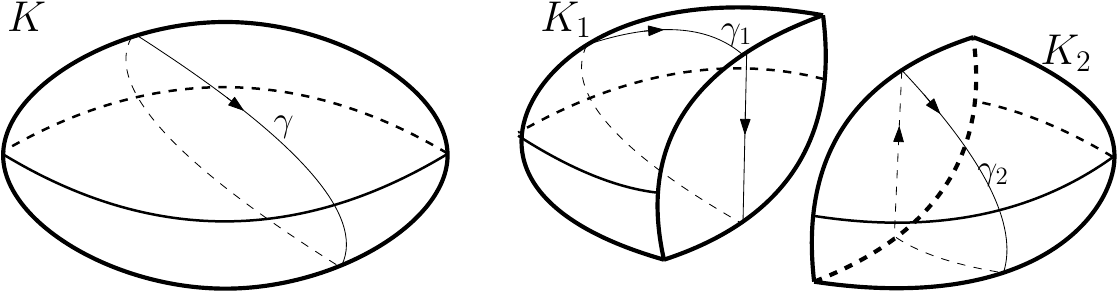}
    \caption{An illustration of a hyperplane cut which cuts an action-minimizing closed characteristic $\gamma$ into action minimizers $\gamma_1,\gamma_2$ for the respective parts.}
    \label{additivityFig}
\end{figure}

As a refinement of Theorem~\ref{ZollIsAdditiveWellDefinedThm}, we establish an implication in the opposite direction. Our next result proves that under certain hypotheses, cuts additivity implies the generalized Zoll property. This can be used to prove the generalized Zoll property for new classes of domains (cf. Example~\ref{ScapQEx}).

Let $K \subset \R^{2n}$ be a convex body. For every $v \in S^{2n-1}$ and $t \in (-h_K(-v), h_K(v))$, let $\Lambda_{v,t}$ be the space of action-minimizers that split into a pair of action minimizers for the respective pieces of a hyperplane cut with normal direction $v$ at height $t$ (cf. Proposition~\ref{additiveDescriptionThm}).

\begin{theorem}[Cuts Additive implies Generalized Zoll] \label{additiveIsZollNonSmoothThm}
    Let $K$ be a cuts additive convex domain. Assume that $K$ satisfies that for every $v \in S^{2n-1}$ and $t \in (-h_K(-v), h_K(v)),$ $\Lambda_{v,t}$ is contractible.
    Then $K$ is generalized Zoll.
\end{theorem}

\begin{remark}
    An additional possible extension of the Zoll property in the non-smooth case is the property that $\partial K$ is foliated by action minimizers, and for each $x \in \partial K,$ the set of action-minimizing closed charactersitcs passing through $x$ is contractible.
    A variant of the proof of Theorem~\ref{additiveIsZollNonSmoothThm} shows that this property implies generalized Zoll.
\end{remark}

\subsection{Non-smooth characteristics}

The characteristic dynamics on a non-smooth convex body $K \subset \R^{2n}$ present substantially more complicated behavior: the space of closed characteristics may be infinite dimensional, and closed characteristics may be very rough---possibly nowhere continuously differentiable.
Our next results show that closed characteristics which are action minimizers exhibit somewhat simpler behavior.

We thus wish to classify the different dynamical behaviors of closed characteristics $\gamma$ at different times $t \in S^1$, and it will be convenient to distinguish three cases. To make this exposition concise, we introduce these properties informally here and give their precise definitions in Section~\ref{nonSmoothSection}.

{\bf (i) Extreme rays:} 
The first property is that $-J\dg(t)$ lies on an extreme ray in the normal cone of $K$ at $\gamma(t).$
We prove in Lemma~\ref{c0Impliesc1Lemma} that, roughly speaking, following $J$ times an extreme ray keeps the velocity stable under small $C^0$ perturbations, making the dynamics resemble the smooth case.
Hence this classification is useful when analyzing dynamical and topological properties of characteristics.

{\bf (ii) Coisotropic faces:} 
The second property is that $\gamma(t)$ moves along a flat face $E$ of $\partial K$ which is coisotropic, and $-J\dg(t)$ is confined to the normal vectors of $E$.
This movement along the face may entail complicated behavior, but in many applications it can be replaced by a straight segment along the face.
In this situation we say that $\gamma$ \emph{moves along the coisotropic face} $E$ at time $t$.

{\bf (iii) Isotropic gliding:}
The last property occurs when $-J\dg(t)$ moves along a time-varying isotropic subspace of the normal cone. We call such a situation \emph{isotropic gliding}. Since $\dg(t)$ is not necessarily continuous, isotropic gliding can arise in a complicated way in many situations, even if, for example, $K$ is a polytope. However, in many cases we do not expect action-minimizing closed characteristics to exhibit isotropic gliding.

The following theorem shows that these three properties completely characterize action-minimizers.

\begin{theorem}[Classification of systoles dynamics]\label{nonSmoothClassificationThm}
    Let $\gamma$ be an action-minimizing closed characteristic. Then, for almost every $t$, either $\dg(t)$ lies on an extreme ray in the normal cone, or $\gamma(t)$ moves along a flat coisotropic face, or $\gamma(t)$ exhibits isotropic gliding.
\end{theorem}

The following proposition shows that, in some situations, action-minimizers do not exhibit isotropic gliding, which makes our classification much simpler. 

\begin{proposition}\label{noIsotropicGlidingThm}
    Let $K \subset \R^{2n}$ be either
        \begin{enumerate}
            \item a convex polytope, or
            \item a convex Lagrangian product.
        \end{enumerate}
        Then $K$ does not admit an action minimizer closed characteristic with isotropic gliding.
\end{proposition}

\begin{remark}
    Without the action-minimizing requirement, convex polytopes and convex Lagrangian products do admit characteristics with isotropic gliding (see Remark~\ref{polytopeGlidingRemark}).
\end{remark}

\medskip \medskip

We now establish a compactness result for the set of action-minimizers.

The space of action-minimizing closed characteristics is contained in the space of absolutely continuous functions with square integrable derivatives, equipped with the $H^1$ topology induced by the natural sobolev norm
\[
\|\gamma\|_{H^1} = \left( \int_0^1 \bigl(\|\gamma(t)\|^2 + \|\dot{\gamma}(t)\|^2\bigr)\, dt \right)^{1/2},
\]
and the associated metric $d_{H^1}(\gamma_1,\gamma_2) := \|\gamma_1 - \gamma_2\|_{H^1}.$
We denote by
$$\|\gamma\|_{C^0} := \max_{t \in S^1} \|\gamma(t)\|$$
the $C^0$ norm. 

In the non-smooth setting, in contrast to the smooth case, the set of action-minimizers may fail to be compact in the $H^1$ topology (see \cite[Proposition 2.1 and Proposition 2.2]{Matijevic}).
Indeed, when moving along a coisotropic face $E$, the velocities of a sequence $\gamma_k$ may oscillate arbitrarily fast between $J$ times different normals of $E$ in such a way that no subsequence converges in the $H^1$ topology. By contrast, this space is compact with respect to the $C^0$ topology.
We show a certain compactness result in the $H^1$ topology after collapsing motions along coisotropic faces.

This is described via an equivalence relation $\sim$ on the set of action-minimizing closed characteristics, and we refer to Definition~\ref{coisotropicCollapseDef} below for the precise formulation. 
Intuitively, two closed characteristics are equivalent if, along coisotropic parts of the boundary, we forget the specific way in which they move inside the coisotropic faces, retaining only the endpoints and the non-coisotropic portions.
We consider the quotient space $Q$ of action-minimizers modulo $\sim$, equipped with the $H^1$ quotient pseudo-metric $d_\sim$ (see Definition~\ref{pseudoMetricDef} below).

\begin{theorem}[$H^1$-compactness of the space of systoles] \label{compactQuotientTopologyThm}
    Let $K$ be a convex domain that does not admit action-minimizers with isotropic gliding.
    Then the quotient space $Q$, equipped with the topology induced by $d_\sim$, is compact. Moreover, this topology coincides with the $C^0$ quotient topology.
\end{theorem}

\begin{remark}
    In the $C^0$ topology each equivalence class is compact, and hence the $C^0$ quotient topology coincides with the topology induced by the $C^0$ quotient pseudo-metric.
\end{remark}

\subsection{Examples}

\begin{example}[The 24-cell]\label{24Ex}
The 24-cell $X$ is the convex polytope with vertices
$$ (\pm1,0,0,0), (0,\pm1,0,0), (0,0,\pm1,0),(0,0,0,\pm1),(\pm\frac{1}{2},\pm\frac{1}{2},\pm\frac{1}{2},\pm\frac{1}{2})$$
and it is one of the six regular polytopes in $\R^4.$ 

In \cite{Ch-H} it is shown that $\rho_{\text{sys}}(X) = 1$. 
It is not known if the interior of $X$ is symplectomorphic to the interior of the ball.

The dual body $X^\circ$ of $X$ is also a 24-cell, in a different position, with vertices $\{\pm e_i \pm e_j\}_{i,j \in \{1,2,3,4\}}$. It also has systolic ratio 1. 
If the interiors of both $X$ and $X^\circ$ are symplectomorphic to the interior of the ball, this gives a particularly symmetric polytopal realization of a symplectic ball in which both a body and its polar are symplectic balls.

As discussed in \cite{Ch-H}, there are no Lagrangian faces in $\partial X$, hence the dynamics is well-defined. Moreover, it is foliated by action-minimizing closed characteristics. Using \cite[Theorem 1.7]{Matijevic} and Theorem~\ref{ZollIsAdditiveWellDefinedThm}, we conclude that $X$ is generalized Zoll and cuts additive.
\end{example}

\begin{example}[Intersection of a simplex and a cube]\label{ScapQEx}
    It is shown in \cite{pazit-thesis-msc} that the standard simplex $S$ maximizes the systolic ratio among simplices in $\R^4$ (which no longer holds in higher dimensions). Its systolic ratio is $\rho_{\text{sys}}(S) = \frac{3}{4}.$

    There are no closed characteristics with minimal action passing through vertices different from the origin. Thus, one can attempt to cut $S$ using hyperplanes parallel to the coordinate axes until a new minimizer appears.
    The resulting domain is 
    $$Y = S \cap [0,\frac{1}{2}]^4,$$
    and it satisfies $\rho_{\text{sys}}(Y) = 1.$
    It is again unknown whether its interior is symplectomorphic to the interior of the ball.
    Note that $Y$ is the positive orthant of the $24$-cell, which also has systolic ratio 1. If both are symplectic balls, this would give an interesting example of a full packing by $16$ balls.

    On $Y$ the characteristic dynamics is not well-defined, and proving that $Y$ is generalized Zoll is more delicate. For instance, the approach in \cite[Proposition~4.2.2]{Matijevic}, which uses a bijection from $\partial Y$ to a subset of $\mathrm{Sys}(Y)$, does not apply here. Indeed, consider points of the form $(0,\varepsilon_1,0,\frac{1}{2}-\varepsilon_2), (\varepsilon_1, 0, \varepsilon_2, \frac{1}{2}) \in \partial Y$. When $\varepsilon_1 > \varepsilon_2 > 0$ the dynamics is well-defined, but the corresponding closed characteristics are genuinely different, and as $\varepsilon_1,\varepsilon_2 \to 0$ they do not converge to the same loop.

    The following proposition uses the combinatorial description of cuts additivity from Lemma~\ref{combinatorialCutLemma} below, together with Theorem~\ref{additiveIsZollNonSmoothThm}.   Its proof involves analyzing different combinatorial cases. It appears in the appendix of the paper, where we show the relevant methods and describe some of these cases.

    \begin{proposition}\label{simplexCapCubeIsZollThm}
        The body $Y$ is generalized Zoll.
    \end{proposition}
    
\end{example}

\begin{example}[The pentagons product]\label{pentagonEx}
    Let $P \subset \R^2$ be the regular pentagon, and let $T$ be its rotation by $90^\circ$. In \cite{counterexample} we showed that
    \[
      \rho_{\text{sys}}(P \times T) = \frac{\sqrt{5}+3}{5} > 1.
    \]
    Numerical experiments suggest that $P \times T$ is cuts exremizer, yet not cuts additive. 
    More precisely, for a specific randomly-picked normal direction $u$, let $K_1(t),K_2(t)$ be the corresponding cut pieces of a hyperplane cut with normal $u$ at height $t$. The way the capacities $\ehzcap(K_1), \ehzcap(K_2)$ and their sum vary with $t$ is depicted in Figure~\ref{fig:pentagonNotAdditiveExample}.
    In particular, this shows that $P \times T$ is not cuts additive and
    $$c_\gamma(P \times T) > \ehzcap(P\times T).$$
    We do not know whether it is generalized Zoll.

    \begin{proposition}\label{pentagonProp}
        Let $v_1,\ldots,v_5$ be the vertices of $P$, and $w_1,\ldots,w_5$ be the corresponding vertices of $T$, so that the segment $[v_i,v_{i+1}]$ is perpendicular to $[w_i,w_{i+1}]$ for every $i$, where indices are considered modulo 5.
        For all directions of the form $u=(u_p, u_q),$ where $u_p \in N_P(v_i)$ and $u_q \in N_T(w_j)$ for $j \in \{i-1,i,i+1\}$, one has that $P \times T$ is cuts additive for hyperplane cuts in direction $u$ in some neighborhood of the boundary. In particular, one has additivity near the boundary for all normals of the form $(u_p,0)$ and $(0,u_q).$
    \end{proposition}
    
\end{example}

\begin{figure}[H]
    \centering
    \includegraphics[width=0.8\linewidth]{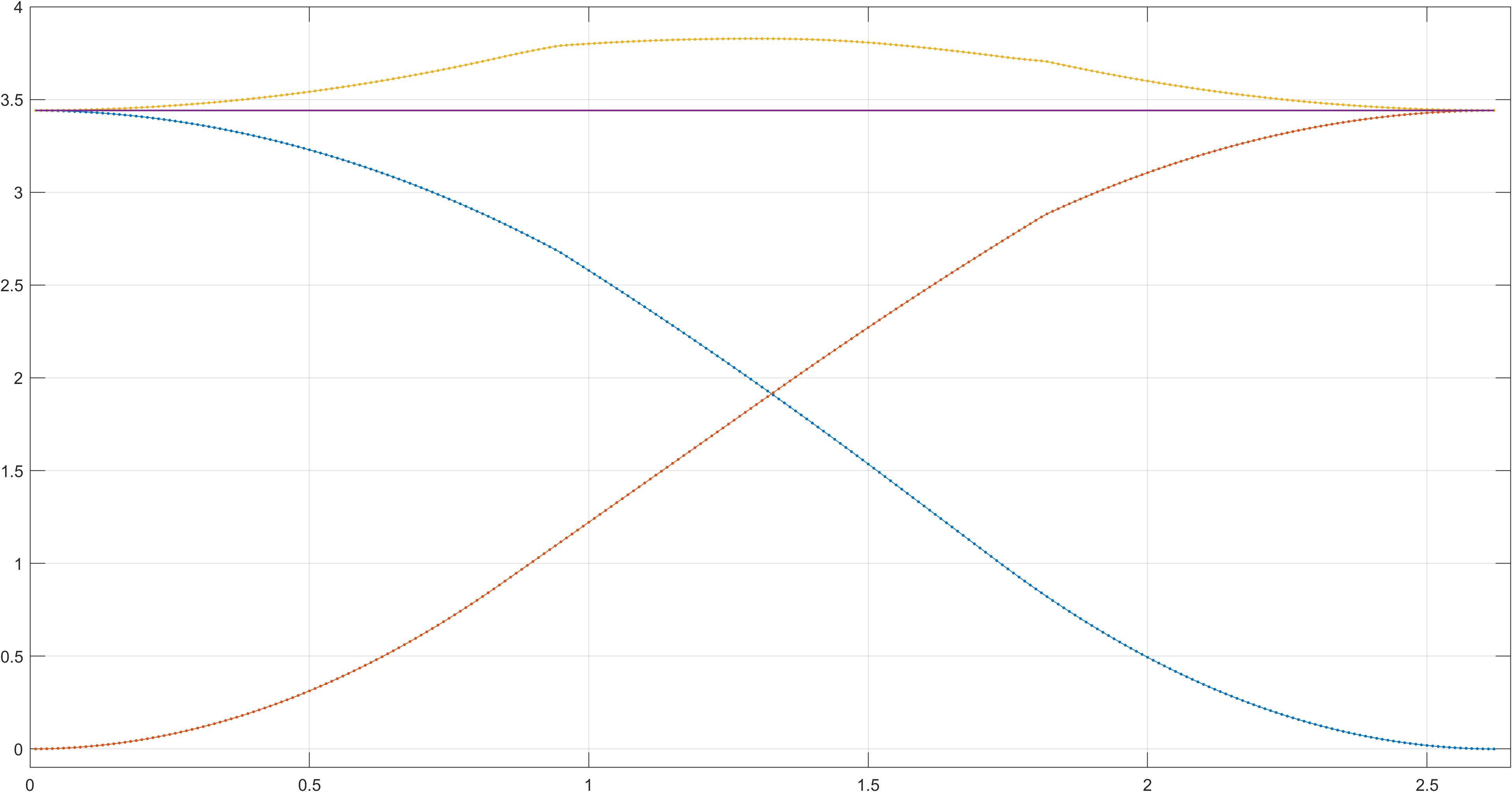}
    \caption{The blue and orange lines describe $\ehzcap(K_2), \ehzcap(K_1)$ respectively, while the yellow line is their sum. The purple line is the capacity of the pentagon product $\ehzcap(P \times T)$. }
    \label{fig:pentagonNotAdditiveExample}
\end{figure}

\begin{example}[Rotation of a cross polytope]
    \label{rotatedCrosspolytopeEx}
    As another symmetric candidate for a symplectic ball, let
    \[
      P = \mathrm{conv}\{\pm e_i\}_{i=1}^4 \subset \R^4
    \]
    be the standard cross-polytope, and set
    \[
    A = \begin{pmatrix}
     1 & 0 & 0 & 0 \\
     0 & \tfrac{\sqrt{2}}{2} & 0 & -\tfrac{\sqrt{2}}{2} \\
     0 & \tfrac{1}{\sqrt{3}} & \tfrac{1}{\sqrt{3}} & \tfrac{1}{\sqrt{3}} \\
     0 & \tfrac{1}{\sqrt{6}} & -\sqrt{\tfrac{2}{3}} & \tfrac{1}{\sqrt{6}}
    \end{pmatrix}.
    \]
    One checks that $AP$ has systolic ratio equal to $1$.

    Denote by $(v_i)_{i=1}^4$ the column vectors of $A$. Then $A$ is a rotation matrix satisfying
    \[
      \omega(v_i, v_j) = \pm \frac{1}{\sqrt{3}}
      \quad \text{for every } i\neq j \in\{1,2,3,4\}.
    \]
    One verifies that there are no Lagrangian faces in $\partial(AP)$ and that $\partial (AP)$ is foliated by action-minimizers. Hence $AP$ is generalized Zoll and cuts additive.

    It is unknown whether this rotation of the cross-polytope is symplectomorphic to the ball.
\end{example}

\begin{remark}
    One can numerically verify that Example~\ref{ScapQEx} satisfy $c_1(Y)=c_2(Y)$ by slightly perturbing the facets and then applying the algorithm from~\cite{Ch-H}. For the perturbed polytope the only closed characteristics with Conley--Zehnder index $5$ have actions extremely close to those of the corresponding minimizer. With additional work, this strategy should yield a rigorous proof that this body is generalized Zoll.
    For perturbations of $P \times T$ in Example~\ref{pentagonEx}, however, this straightforward approach breaks down. In that case there are six essentially distinct action levels for closed characteristics of Conley--Zehnder index $5$, and one of them lies very close to the minimal action. As a result, the same numerical procedure does not provide a clear guess as to whether $P \times T$ is generalized Zoll.

\end{remark}

\subsection{Further discussion on the Zoll property}

An additional related property of convex domains which may be regarded as an extension of the Zoll property to non-smooth domains is the coincidence of $\ehzcap(K)$
and the diameter $c_{\gamma}(K)$ of the spectral norm. 
The spectral norm of a Hamiltonian diffeomorphism is defined via a homological min-max process on filtered Floer homology. It was introduced by Viterbo \cite{Viterbo-specGF} via generating functions, while the Floer-theoretic counterparts were developed in \cite{Schwarz2000, Oh2005, oh-spectral}. The \emph{diameter of the spectral norm} 
\[
c_{\gamma}(K) := \sup\{\gamma(\phi_H)\}
\]
over all Hamiltonian isotopies $H$ supported in $K$, was first considered in \cite{frauenfelderSchlenk}, and studied as a symplectic capacity in \cite{AAC}
 where, using the coincidence between the symplectic homology capacity and the EHZ capacity which was proved independently in \cite{abbondandolo-kang} and \cite{irie}, they showed that the ball $B^{2n}(1)$ satisfies
\[
\ehzcap(B^{2n}(1)) = c_{\gamma}(B^{2n}(1)).
\]
This observation suggests an additional extension of the Zoll property in the non-smooth setting.

The following statement about a (not necessarily smooth) convex body is a simple consequence of Theorem~1 in \cite{cant} together with the subadditivity from~\cite{pazit} and we omit its proof.

\begin{thm}[\cite{cant, pazit}]\label{c_gamma_implies_additive_thm}
    Let $K \subset \R^{2n}$ be convex with $\ehzcap(K) = c_{\gamma}(K)$. Then $K$ is cuts additive.
\end{thm} 

\medskip \medskip

We have the following conjectural picture in the smooth case.
\begin{conjecture} \label{smoothMaximizerEquivalenceConj}
    Let $K \subset \R^{2n}$ be smooth and convex, then the following conditions are equivalent.
    \begin{enumerate}
        \item $K \subset \R^{2n}$ is symplectomorphic to the ball.
        \item $K$ is a local maximizer of the systolic ratio with respect to $C^2$ perturbations.
        \item $K$ is a local maximizer of the systolic ratio with respect to the Hausdorff distance.
        \item $\partial K$ is foliated by action minimizers.
        \item $K$ is cuts additive.
        \item $K$ is cuts extremizer.
        \item $\ehzcap(K) = c_{\gamma}(K)$.
        \item $\ehzcap(K) = c_n(K).$
        \item The FR index of the action-minimizers is $n$.
    \end{enumerate}
\end{conjecture}

\begin{remark}
    It is known that Conjecture~\ref{smoothMaximizerEquivalenceConj}, apart from Property 3, holds in $\R^4$ for strictly convex domains.
    
    Indeed, 
    Monotonicity implies that a $C^2$ local maximizer must be a cuts extremizer, which by Corollary~\ref{extremizerIsAdditiveCor} yields cuts additivity and the Zoll property.
The fact that Zoll domains are local maximizers with respect to the $C^2$ topology in all dimensions is shown in \cite{ABHS, abbondandoloBenedetti, abbondandolo-benedetti-edtmair}. 

It is shown in \cite{Gi-Gu-Ma} that $K$ is Zoll if and only if $\ehzcap(K) = c_n(K)$, and that this happens if and only if the FR index of the action-minimizers equals $n$ (see also \cite{Matijevic}). 

Theorem~\ref{c_gamma_implies_additive_thm} implies that $\ehzcap(K) = c_\gamma(K)$ forces cuts additivity.

    It is known that all of the above properties hold for the ball, and all other implications are due to the result in \cite{ABHS} which states that in dimension four, every Zoll convex domain is symplectomorphic to the ball.

    With the exception of Properties 1, 3, and 7, all other properties remain equivalent in higher dimensions.
\end{remark}

We remark that Conjecture~\ref{smoothMaximizerEquivalenceConj} is false without the smoothness condition. 
If $K$ is a local maximizer of $\rho_{\text{sys}}$ it does not imply that it is a symplectic ball; as discussed above, being foliated by action minimizers does not imply generalized Zoll; and Example~\ref{pentagonEx} shows that being cuts extremizer does not imply cuts additivity. This suggests a more nuanced picture in the non-smooth case.

\medskip \medskip

\noindent {\bf Structure of the paper:}
In Section~\ref{nonSmoothSection} we discuss the dynamical behaviors of action-minimizers in the non-smooth setting, and prove Theorem~\ref{nonSmoothClassificationThm}, Proposition~\ref{noIsotropicGlidingThm}, and Theorem~\ref{compactQuotientTopologyThm}.
Section~\ref{dynCutSection} is concerned with the dynamical characterization of cuts additivity, in which we prove Proposition~\ref{additiveDescriptionThm} and show, in Lemma~\ref{combinatorialCutLemma}, a combinatorial description of cuts additivity for polytopes.
In Section~\ref{additiveZollSection} we discuss the equivalence of the cuts additivity and generalized Zoll properties, and prove Theorem~\ref{ZollIsAdditiveWellDefinedThm} and Theorem~\ref{additiveIsZollNonSmoothThm}.
Finally, in the appendix of the paper, we prove Proposition~\ref{simplexCapCubeIsZollThm} and Proposition~\ref{pentagonProp}.

\medskip \medskip

\noindent{\bf Acknowledgment:} 
I wish to thank Alberto Abbondandolo, Yaron Ostrover, and Shira Tanny for many helpful suggestions. 
This work was partially supported by a grant from the Institute for Advanced Study by the Erik Ellentuck Fellow Fund.

\section{Non-smooth action minimizers}\label{nonSmoothSection}

In this section, we give the precise definitions of moving along coisotropic faces and isotropic gliding, and then prove Theorem~\ref{nonSmoothClassificationThm}, Proposition~\ref{noIsotropicGlidingThm}, and Theorem~\ref{compactQuotientTopologyThm}. 

\medskip

Let us start with some preliminaries on generalized characteristics (we refer to e.g. \cite{Kunzle, singular_survey_1} for more details).
Let $K \subset \R^{2n}$ be convex. For $x \in \partial K$ we define the \emph{normal cone} of $K$ at $x$ by
\[
 N_K(x) = \{ v \in \R^{2n} : \langle v, x-y \rangle \geq 0 \ \forall\, y\in K \}.
\]
When $\partial K$ is smooth, $N_K(x) = \R_+ u$, where $u$ is the outer unit normal to $\partial K$ at $x$.

A \emph{characteristic} in the non-smooth setting is a curve
\[
\gamma \in W^{1,2}([0,T], \partial K), \quad \gamma(0)=\gamma(T),
\]
such that
\[
 -J \dot{\gamma}(t) \in N_K(\gamma(t))
\]
for almost every $t\in[0,T]$. 
Recall that $W^{1,2}$ denotes the Hilbert space of absolutely continuous functions with square-integrable derivative. 

Throughout the paper we normalize $K$ so that $\ehzcap(K) = 1$, and we normalize characteristics by requiring
$ \dg(t) \in J \partial (\|\cdot\|_K^2)$
for almost every $t$, where 
$$\|\cdot\|_K := \inf\{ \lambda : x \in \lambda K\}$$ 
is the norm associated with the convex body $K$, and $\partial (\|\cdot\|^2_K)$ is the subdifferential (see \cite{subdiff_1}). Equivalently $h_K(-J \dg(t)) = 2$, where 
$$h_K(v) = \sup\{ \langle x,v\rangle:x \in K\}$$
is the support function of $K$. This normalization requires translating $K$ so that the origin lies in its interior. Under this normalization, the action of a loop coincides with its period, and we view closed characteristics with minimal action as loops $\gamma: S^1 \to \partial K$ of period $1$, where $S^1 = \R / \mathbb{Z}$.
The space $\mathrm{Sys}(K)$ of normalized action-minimizing closed characteristics is now naturally endowed with an $S^1$ action given by time translations. 

For $x \in \partial K$, the characteristic direction is contained in
    \[
    D_K(x) := J N_K(x) \cap T_K(x),
    \]
    where
    \[
      T_K(x) := \{ v \in \R^{2n} : \langle v, u\rangle \leq 0 \ \forall u \in N_K(x)\}
    \]
    is the tangent cone of $K$ at $x$. Note that $T_K(x)$ describes those directions that locally remain inside $K$.

\medskip

We now recall Clarke's dual action principle, and introduce a version of Clarke functional that would enable us to consider Clarke's duality even if the origin is not in the interior of the body, a useful fact when analyzing hyperplane cuts.
We consider the minimizers of the Clarke functional
\[
I^p_K: W^{1,2}(S^1,\R^{2n}) \to \R, \quad
I^p_K(\gamma) := \left( \frac{1}{2^p}\int_0^1 h_K(\dot{\gamma}(t))^p\,dt\right)^{\frac{2}{p}}
\]
under the constraint $\mathcal{A}(\gamma) = 1$.  
For every minimizer $\gamma$ there exists a translation $x_\gamma \in \R^{2n}$ such that $\sqrt{\ehzcap(K)} \gamma + x_\gamma \in \mathrm{Sys}(K)$, and $I_K^p(\gamma) = \ehzcap(K)$ for all $p \geq 1$ (see \cite{ArtsteinAvidanOstrover2008}).

    We denote $I_K(\gamma) := I_K^2(\gamma)$ and, for invariance under translations, we consider $\Phi_K(\gamma) := (I^1_K(\gamma))^\frac{1}{2}.$
    Note that for any $\gamma \in W^{1,2}(S^1,\R^{2n})$, not necessarily an action-minimizer, we have
    $$ \ehzcap(K) \leq \frac{\Phi(\gamma)^2}{\mathcal{A}(\gamma)}$$
    We observe that when translating $K$ by $\widetilde{K} = K + x_0$ for $x_0 \in \R^{2n}$, 
    $$\int_0^1 h_{\widetilde{K}}(-J\dg(t))dt = \int_0^1 h_K(-J\dg(t)) + \langle x_0, -J \dg(t) \rangle dt = \int_0^1 h_K(-J \dg(t))dt.$$
    So this minimization problem is invariant under translations, and in particular we do not need the assumption that the origin is in the interior of $K$, albeit then we can no longer normalize so that $h_K(\dg)$ is constant.

    \medskip \medskip

    We now turn to discuss the dynamical behavior of action-minimizers.
    We call $E = \widetilde{E} \cap N_K(x)$, where $\widetilde{E}$ is a linear subspace of $\R^{2n},$ an \emph{extreme face} of $N_K(x)$ if for every $u \in E$, there are no $v_1,v_2 \in N_K(x) \setminus E$ and $a \in (0,1)$ such that
\[
u = a v_1 + (1-a) v_2.
\]
When $\widetilde{E}$ is one dimensional, we call $E$ an extreme ray.

The following result can be viewed as an extension of the work \cite{pazit} to the case where $K$ is not necessarily a polytope.
\begin{thm} \label{extremeOrIsotropicThm}
    Let $\gamma \in \mathrm{Sys}(K)$. For almost every $t \in S^1$, 
    if $\dg(t)$ is a non-trivial linear combination of elements in $J N_K(\gamma(t))$, then their pairwise $\omega$-values vanish.
\end{thm}
\begin{remark}
    We remark that Theorem~\ref{extremeOrIsotropicThm} does not hold if one removes the assumption that $\gamma$ is an action-minimizer. For instance, the velocity $\dg(t)$ is not contained in an extreme ray of $J N_K(\gamma(t))$ when $\gamma(t)$ lies on an edge of a four-dimensional polytope which is a \emph{bad 1-face} in the sense of~\cite{Ch-H}.
\end{remark}

Let us proceed to the proof of Theorem~\ref{extremeOrIsotropicThm}.
The following lemma roughly states that outside movements on $J$ times isotropic faces of the normal cone, the velocity of $\gamma$ is always on an extreme ray.
\begin{lemma} \label{nonExtremeIsIsotropicLemma}
    Let $\gamma \in \text{Sys}(K)$, and let $u_i(t) \in N_K(\gamma(t)), i=1,\ldots,k$ be measurable selections normalized by $h_K(u_i(t))  = 2$, and satisfying that there exists measurable $a_i(t) : S^1 \to \R, i=1,\ldots,k$ with $0 < a_i(t) < 1$ and with $-J \dot{\gamma}(t) = \sum a_i(t) u_i(t)$ almost everywhere. Then one has for almost every $t$ and all $i\neq j$, $\omega(u_i(t),u_j(t)) = 0.$ 
\end{lemma}
\begin{remark}
    There always exist measurable selections $u_i(t),a_i(t)$ as above.
    Note that if $\dot{\gamma}(t)$ is along an extreme ray, then one is able to put $u_i(t)=-J\dot{\gamma}(t)$ for every $i$.
\end{remark}
\begin{remark}
    The condition $h_K(u_i(t)) = 2$ is consistent with our normalization $h_K(-J \dg(t)) = 2,$ since $h_K$ restricted to the same normal cone is a linear function.
\end{remark}
\begin{proof}
    Without loss of generality, we may assume that $\ehzcap(K) = 1.$
    For simplicity, we assume that $k=2$, and put 
    $$\dg(t) = a(t) J u(t) + (1-a(t)) J v(t).$$ 
    The proof for $k >2$ is similar.
    Recall that we normalize $\gamma$ so that $h_K(-J\dot{\gamma}(t)) = 2$ for almost every $t$, so $\gamma$ is of period 1, and $\mathcal{A}(\gamma) = 1.$
    In addition, by Clarke's duality, every 
    $\gamma': S^1 \to \R^{2n}$ 
    with $h_K(-J\dot{\gamma}'(t)) = 2$ almost everywhere, satisfies 
    $$\mathcal{A}(\gamma') \leq 1.$$ 
    Let $t_0 \in S^1$ and let $\varepsilon > 0.$
    Roughly speaking, we will show that for a generic choice of $t_0$ with $\omega(u(t_0), v(t_0)) \neq 0$, one of the loops that moves $\varepsilon$-time either in $u$ and then in $v$ or in $v$ and then in $u$ instead of their convex combination, has larger action which would contradict the minimality of $\gamma$.
    For $s \in (-\varepsilon, \varepsilon),$ denote 
    $$f(s) = \int_{-\varepsilon}^s a(t_0+\theta)d\theta, \;\;\quad g(s) = s+\varepsilon - f(s).$$
    Denote $\tau_1 = f(\varepsilon),$
    $\tau_2 = g(\varepsilon).$
    Since $f,g$ are monotone increasing, they have inverses $f^{-1}(s), g^{-1}(s)$ defined for $s \in (0, \tau_1)$ and $s \in (0,\tau_2)$ respectively.
    Let $\gamma_1$ be defined by 
    \[
\dot{\gamma}_1(t) =
\begin{cases}
  \dot{\gamma}(t), & t \in (0,\,t_0 - \varepsilon), \\[6pt]
  J u\!\bigl(t_0+f^{-1}(t - t_0 + \varepsilon)\bigr), & t \in (t_0 - \varepsilon,\,t_0 - \varepsilon + \tau_1), \\[6pt]
  J v\!\bigl(t_0+g^{-1}(t - t_0 - \tau_1 + \varepsilon)\bigr), & t \in (t_0 - \varepsilon + \tau_1,\,t_0 + \varepsilon), \\[6pt]
  \dot{\gamma}(t), & t \in (t_0 + \varepsilon,\,1),
\end{cases}
\]
    and $\gamma_2$ by
    \[
\dot{\gamma}_2(t) =
\begin{cases}
  \dot{\gamma}(t), & t \in (0,\,t_0 - \varepsilon), \\[6pt]
  J v\!\bigl(t_0+g^{-1}(t - t_0 + \varepsilon)\bigr), & t \in (t_0 - \varepsilon,\,t_0 - \varepsilon + \tau_2), \\[6pt]
  J u\!\bigl(t_0+f^{-1}(t - t_0 - \tau_2 + \varepsilon)\bigr), & t \in (t_0 - \varepsilon + \tau_2,\,t_0 + \varepsilon), \\[6pt]
  \dot{\gamma}(t), & t \in (t_0 + \varepsilon,\,1),
\end{cases}
\]
    Since $u,v$ were chosen so that $h_K(u) = h_K(v) = 2,$ we have
    $$\frac{1}{4}\int_0^1 h_K(-J\dg(t))^2dt = \frac{1}{4}\int_0^1 h_K(-J\dg_1(t))^2dt = \frac{1}{4}\int_0^1 h_K(-J\dg_2(t))^2dt = 1.$$
    With respect to action, we have that
    $$\mathcal{A}(\gamma) - \mathcal{A}(\gamma_1) = \frac{1}{2} \int_{t_0-\varepsilon}^{t_0+\varepsilon} \int_{t_0-\varepsilon}^{t} a(t)(1-a(s))\omega(u(t),v(s))dsdt,$$
    $$\mathcal{A}(\gamma_2) - \mathcal{A}(\gamma) = \frac{1}{2}\int_{t_0-\varepsilon}^{t_0+\varepsilon} \int_{t_0-\varepsilon}^{t} a(s)(1-a(t))\omega(u(s),v(t))dsdt.$$
    Denote $Y(t,s) := a(t) (1-a(s)) \omega(u(t),v(s)).$
    For $t_0$ a Lebesgue point for $a, u, v$, we have
    \begin{align*} 
    &\frac{1}{2\varepsilon} \int_{t_0-\varepsilon}^{t_0+\varepsilon} |a(t) - a(t_0)| dt = o(1), \;\; \frac{1}{2\varepsilon} \int_{t_0-\varepsilon}^{t_0+\varepsilon} \|u(t) - u(t_0)\| dt = o(1), \\ 
    &\frac{1}{2\varepsilon} \int_{t_0-\varepsilon}^{t_0+\varepsilon} \|v(t) - v(t_0)\| dt = o(1). 
    \end{align*}
    Since $h_K(u(t)) = h_K(v(t)) = 2$ and $a(t) \in (0,1)$, all functions are uniformly bounded which automatically yields
    $$ \frac{1}{(2\varepsilon)^2} \int_{t_0-\varepsilon}^{t_0+\varepsilon} \int_{t_0-\varepsilon}^{t_0+\varepsilon} |Y(t,s) - Y(t_0,t_0)| = o(1).$$
    Plugging this into the formula for the action gives
    $$ \mathcal{A}(\gamma) - \mathcal{A}(\gamma_1) = \frac{1}{2} \int_{t_0-\varepsilon}^{t_0+\varepsilon} \int_{t_0-\varepsilon}^{t} Y(t_0,t_0) + (Y(t,s) - Y(t_0,t_0))ds dt = \varepsilon^2 Y(t_0,t_0) + o(\varepsilon^2).$$
    Similarly,
    $$ \mathcal{A}(\gamma_2) - \mathcal{A}(\gamma) =\varepsilon^2 Y(t_0,t_0) + o(\varepsilon^2).$$
    The extremality of $\gamma$ gives $\mathcal{A}(\gamma) - \mathcal{A}(\gamma_1)\geq 0, \; \; \mathcal{A}(\gamma_2) - \mathcal{A}(\gamma) \leq 0,$ which implies, as $\varepsilon \to 0$, $Y(t_0,t_0) = 0.$

    \end{proof}

    \begin{proof}[Proof of Theorem~\ref{extremeOrIsotropicThm}]
        Assume by contradiction, that in a positive measure set $\dg(t) = \sum a_i(t) J u_i(t)$ for some vectors on extreme points of $N_K(\gamma(t)) \cap S$, $$u_i(t) \in N_K(\gamma(t)) \cap S, \;\;\; a_i(t) \in (0,1), \;\;\; i=1,\ldots,2n$$ 
        where $S = \{ u \in \R^{2n}: h_K(u)=2\}$, and there exists $i,j \in \{1,\ldots,2n\}$ so that $\omega(u_i(t),u_j(t)) \neq 0.$
        In fact, as an application of the Jankov--von Neumann uniformization theorem, one can choose $u_i(t), a_i(t)$ to be measurable. Morever, they can then be extended to measurable selections for all times.
        Hence we have selections of extreme points $$ u_i : S^1 \to N_K(\gamma(t)) \cap S, \qquad a_i:S^1 \to (0,1),$$
        with $\dg(t) = \sum a_i(t) J u_i(t)$, satisfying that on a positive measure set, there exists $i\neq j$ with $\omega(u_i(t),u_j(t)) \neq 0,$ which contradicts Lemma~\ref{nonExtremeIsIsotropicLemma}.

    \end{proof}
    
\medskip\medskip

Let us now give the precise definition of the distinction between moving along coisotropic faces and isotropic gliding.
Let $\gamma \in \mathrm{Sys}(K)$.
Let $E$ be a \emph{coisotropic face} of $\partial K$. I.e. $E$ is an extreme convex subset of $\partial K,$  and $TE = \text{span}(E-E)$ is coisotropic.
In that case, we have $D_K(x) \subset TE$ for every $x \in \text{relint}(E).$
We denote by $N_K(E) = \bigcap_{x \in E} N_K(x).$ Equivalently, $N_K(E) = N_K(x)$ for $x \in \text{relint}(E).$ Note that $N_K(E)$ lies on an isotropic subspace.
\begin{definition}\label{coisotropicFaceDef}
    We say that $\gamma \in \text{Sys}(K)$ \emph{moves along the coisotropic face} $E$ at time $t$ if there exists an open interval $t \in (t_0,t_1) \subset S^1$ with $\gamma(s) \in E$ and 
$\dg(s) \in J N_K(E)$
for almost every $t_0 < s < t_1.$
We will say that $\gamma$ \emph{strictly moves along} $E$ at time $t$, if there exists an open interval as above which satisfies in addition $D_K(\gamma(s))  \subset TE$ for all $t_0 < s < t_1.$
\end{definition}

\begin{remark}
    If $\gamma \in \text{Sys}(K)$ moves along the coisotropic face $E$ at times $(t_0,t_1)$, then every other closed characteristic $\gamma'$ which moves along $E$ at times $(t_0,t_1)$ and agrees with $\gamma$ outside $(t_0,t_1)$ is also an action-minimizer. 
\end{remark}

\begin{definition} \label{isotropicGlidingDef}
Let $\gamma \in \mathrm{Sys}(K)$. We say that a positive measure set $A \subset S^1$ is an \emph{isotropic gliding set} if:
\begin{itemize}
    \item For every $t \in A$, $\dot{\gamma}(t)$ lies in the relative interior of an isotropic subspace of $J N_K(\gamma(t))$ of dimension greater than one.
    \item For every $t \in A$, $\gamma$ does not move along any coisotropic face of $\partial K$ at time $t$ (in the sense defined above).
\end{itemize}
\end{definition}

\begin{proof}[Proof of Theorem~\ref{nonSmoothClassificationThm}]
This is an immediate consequence of Theorem~\ref{extremeOrIsotropicThm}.
Specifically, for almost every $t$, $-J\dg(t)$ belongs either to an extreme ray or to an isotropic subspace of $N_K(\gamma(t))$.
In the latter case, $\gamma(t)$ may move along a coisotropic face.
If the set of times at which $\gamma(t)$ does not move along a coisotropic face and $-J\dg(t)$ does not belong to an extreme ray has positive measure, then it is an isotropic gliding set.
\end{proof}

\medskip \medskip

We now move to the proof of Proposition~\ref{noIsotropicGlidingThm}.
Let us first make the following observation regarding action-minimizing characteristic dynamics, which will be useful later.

\begin{definition}
    Let $K \subset \R^{2n}$ be a convex body, and let $\gamma$ be a characteristic.
    Let a normal $u \in N_K(\gamma(t))$ be on an extreme ray of $N_K(\gamma(t)).$
    We say that $u$ is $\gamma$-\emph{active} at time $t$ if for a small interval $I$ surrounding $t$, one can write $\dg(s) = a(s)Ju + J\nu(s)$ for almost every $s \in I$ such that
    \begin{itemize}
        \item $\nu(s) \in N_K(\gamma(s))$.
        \item $a(t) > 0.$
        \item $a(s) \geq 0$ and $a(s)>0$ only if $u \in N_K(\gamma(s))$.
        \item $a$ and $\nu$ are measurable.
        \item $t$ is a Lebesgue point for $a$ and for $\nu$.
    \end{itemize}
    If there is no ambiguity with the choice of closed characteristic, we say that $u$ is \emph{active} at time $t$.
    When there is a set of normals $\{n_{1},\ldots,n_{m}\} \subset N_K(\gamma(t))$ with $\dg(s) = \sum_{j=1}^m a_j(s) n_{j}$, and $a_j(t)>0$ (defined similarly with $t$  a Lebesgue point for $a_j, j=1,\ldots,m$), we call $\{n_{1},\ldots,n_{m}\}$ an \emph{active set} at time $t$.
    Note that by Theorem~\ref{extremeOrIsotropicThm}, when $\gamma \in \text{Sys}(K)$ each active set spans an isotropic subspace. Moreover, the space of such linear combinations is convex, and we get that the $\omega$-value vanishes for every pair of active normals at time $t$.
    The set that includes all active normals is called the \emph{maximal active set} at time $t$.
\end{definition}

\begin{lemma} \label{normalsCrossLemma}
Let $K \subset \R^{2n}$ be a convex body, and let $\gamma \in \text{Sys}(K).$
Suppose that $u \in \R^{2n}$ is active at times $t_1$ and $t_2$ with $t_1< t_2.$
Then 
$$ \int_{t_1}^{t_2} \omega(u, -J\dg(t)) dt = 0. $$
    
\end{lemma}

\begin{proof}
    Let $a: S^1 \to \R$ be a choice for a measurable function describing the coefficient of $u$ in $\dg.$
    Meaning, $a(t) \geq 0$ for every $t \in S^1,$ $a(t)>0$ only if $u \in N_K(\gamma(t)),$ and $\dg(t) = a(t) Ju + J\nu(t),$ for almost every $t \in S^1,$ where $\nu(t) \in N_K(\gamma(t)).$
    Choose small intervals $I_1, I_2$ surrounding $t_1,t_2$ respectively. Denote $\varepsilon = \max\{|I_1|,|I_2|\}$ and $\tau_1 = \int_{I_1} a(t)dt , \; \tau_2 = \int_{I_2} a(t)dt.$ Note that $\tau_1,\tau_2 > 0$.
    Let $\delta>0$ be small enough so that $ \tau_1\delta,  \tau_2\delta < 1.$
    Let $\gamma_1$ be defined by 
    $$ \dg_1(t) = 
    \begin{cases}
  \dot{\gamma}(t), & t \not\in I_1 \cup I_2, \\[6pt]
  (1-\tau_2 \delta)a(t) Ju + J \nu(t), & t \in I_1,\\[6pt]
  (1+\tau_1 \delta ) J u + J\nu(t), & t \in I_2
\end{cases}
$$
    Similarly, $\gamma_2$ is defined by 
    $$\dg_2(t) = 
    \begin{cases}
    \dg(t), & t \not\in I_1 \cup I_2,\\[6pt]
    (1+\tau_2 \delta)a(t) Ju + J \nu(t), & t \in I_1,\\[6pt]
    (1-\tau_1 \delta) J u + J\nu(t), & t \in I_2
        
    \end{cases}$$
    The values of the linear Clarke functional $\Phi_K(\gamma_i), i=1,2$ coincide with $\Phi_K(\gamma)$, and the respective action differences from $\gamma$ are 
    $$ \tau_1 \tau_2 \delta \left( \pm\int_{t_1}^{t_2} \omega(u, -J\dg(t)) + o(\varepsilon)\right)$$ 
    as $|I_1|,|I_2| \to 0.$
    By a standard Clarke's duality argument, we get the required.
    
    We note that one has an $o(\varepsilon)$ term and not just $o(1)$ term above, since Theorem~\ref{extremeOrIsotropicThm} guarantees that $\omega(u, \dg(t)) = 0$ for almost every $t$ with $a(t) > 0.$
    We remark that $o(1)$ is sufficient for our purposes, but the $o(\varepsilon)$ estimate is needed for the proof of Lemma~\ref{finiteNormalsLemma}.
    
\end{proof}

\begin{lemma}\label{finiteNormalsLemma}
Let $K \subset \R^{2n}$ be a convex body, and let $\gamma \in \text{Sys}(K).$
Suppose that $u,v \in \R^{2n}$ with $\omega(u,v) \neq 0$ satisfy the following. The normal $u$ is active at times $t_1,t_3$ and $v$ is active at $t_2$, with $t_1<t_2<t_3.$ Then $v$ is not active for every $t \not\in (t_1,t_3).$
    
\end{lemma}

\begin{proof}
    Argue by contradiction, and assume that there is $t_4 > t_3$ for which $v$ is active.
    As in the proof of Lemma~\ref{normalsCrossLemma}, using Clarke's duality, one can increase the coefficients of $u$ near $t_3$ and reduce them near $t_1$. Simultaneously we either increase the coefficients of $v$ near $t_4$ and reduce them near $t_2,$ or vice versa.
    The proof of Lemma~\ref{normalsCrossLemma} yields that the leading terms in the action differences are now $\pm c \omega(u,v)$ for some constant $c$, and repeating the same arguments as before, we get $\omega(u,v) = 0,$ contradicting the assumptions.

\end{proof}

\begin{definition}
    Let $K \subset \R^{2n}$ be a convex polytope, and let $\gamma \in \text{Sys}(K).$
    Note that the extreme rays of $N_K(\gamma(t))$ are among rescalings of outer normal vectors to the facets of $K$. 
    Denote by $F_1,\ldots, F_N$ the faces of $K$.
    Denote $A_{F} := \{ t: \gamma(t) \in \text{relint}(F)\}$ the set of times for which $\gamma$ is in the relative interior of a face $F$ of $\partial K$. Note that $A_{F_i}$ is a disjoint cover of $S^1.$
    We call $t \in S^1$ a \emph{face transition} if $\forall \delta >0,$ there are $i \neq j$ and $t_1,t_2 \in (t-\delta, t+\delta)$ with $t_1 \in A_{F_i}$ and $t_2 \in A_{F_j}.$
\end{definition}

\begin{corollary}\label{finiteFacesLemma}
    Let $K \subset \R^{2n}$ be a convex polytope, and let $\gamma \in \text{Sys}(K).$
    Then the number of face transitions is finite.
\end{corollary}

\begin{proof}
    Denote $\mathcal{N} := \{n_1, \ldots, n_k\}$ to be the unit outer normals to the facets of the polytope.
    Let $F = \{x \in \partial K : h_K(n_{i_j}) = \langle x,n_{i_j} \rangle, \; j=1,\ldots,m \}$ denote a face of $K$ which is the intersection of the supporing hyperplanes with normals $\{n_{i_1}, \ldots, n_{i_m}\}$.
    It is enough to prove that for each $F$, the set $\{ t\in S^1 : \gamma(t) \in F\}$ is a finite union of connected sets.

    Assume that $\gamma(t^{(l)}_1), \gamma(t^{(l)}_3) \in F$, and $\gamma(t^{(l)}_2) \not\in F,$ where $l \in \N$, and 
    $$t^{(1)}_1 < t^{(1)}_2 < t^{(1)}_3<t^{(2)}_1<t_2^{(2)}<t_3^{(2)}<\ldots.$$
    It is straightforward to check that there exists an active normal $v^{(l)} \in \mathcal{N}$ for some $t \in (t^{(l)}_1,t^{(l)}_2)$ and a normal $w^{(l)}_1 \in \mathcal{N} \cap N_K(F)$ with $\langle Jv^{(l)}, w^{(l)} \rangle < 0.$ Then there must be a normal $u^{(l)} \in \mathcal{N}$ which is active for some $t \in (t^{(l)}_2,t^{(l)}_3)$ and satisfies $\langle J u^{(l)}, w^{(l)} \rangle > 0.$
    Since there are a finite number of vectors in $\mathcal{N},$ there is some $u,v,w$ so that $u=u^{(l)}, v=v^{(l)}, w=w^{(l)}$ for infinite values of $l$, which contradicts Lemma~\ref{finiteNormalsLemma}. 
    
\end{proof}

\medskip

We are now ready to prove Proposition~\ref{noIsotropicGlidingThm}.

    \begin{proof}[Proof of Proposition~\ref{noIsotropicGlidingThm}]
        Suppose first that $K$ is a convex polytope, and let $\mathcal{N}$ be the set of normals to the facets of the polytope.
        Assume that $A \subset S^1$ is an isotropic gliding set for a characteristic $\gamma$ on $\partial K.$
        By the definition of $A$, there exist active sets with at least two normals for almost every $t \in A$.
        By Lemma~\ref{finiteFacesLemma}, for almost every $t \in A$, there is a face $F$ of $\partial K$, where $\gamma((t-\delta, t+\delta)) \subset \text{relint}(F)$ for some sufficiently small $\delta > 0.$
        If $\dg(s)$ is constant on $(1-\delta,1+\delta),$ by the definition of $A$, $\dg(s) \equiv \sum_{i=1}^m a_i n_i$ with $n_i \in \mathcal{N},$ $a_i > 0$, and $m > 1.$ Take 
        $$E = \{ x \in \partial K : h_K(n_i) = \langle x, n_i \rangle, \; i=1,\ldots,m \}$$ 
        to get $\gamma((t-\delta, t+\delta)) \subset E,$ and $\dg((t-\delta, t+\delta)) \subset J N_K(E).$
        Otherwise, let $$\mathcal{Q} :=  N_K(F) \cap N_K(F)^\omega = \{v \in N_K(F) : \omega(u,v) = 0, \; \forall u\in N_K(F)\}.$$ 
        From the fact that $\omega(v,-J\dg(s)) = \langle v, \dg(s) \rangle = 0,$ for all $v \in N_K(F),$ and $-J \dg(s) \in N_K(F)$ one has $-J\dg(s) \in \mathcal{Q}$ for almost every $s \in (t-\delta, t+\delta).$
        Let $$E = \{x \in \partial K: h_K(v) = \langle x, v \rangle \; \forall v\in \mathcal{Q}\}.$$
        One has $\mathcal{Q} \subset N_K(E)$, and hence $-J \dg((t-\delta, t+\delta)) \subset N_K(E).$
        Since $\mathcal{Q} \subset N_K(F)$, we have $F \subset E,$ and hence $\gamma((t-\delta, t+\delta)) \subset E.$ 
        Since $\dg(s)$ is not constant over $(t-\delta,t+\delta)$, $\dim N_K(E) >1,$ which completes the proof.

        \medskip
    
        Next assume that $K = T_1 \times T_2 \subset \R^n_q \times \R^n_p$ is a Lagrangian product, and $A \subset S^1$ is an isotropic gliding set for an action-minimizer $\gamma \in \text{Sys}(K).$
        From the billiard characterization of characteristics on Lagrangian products \cite{AA-O} we know that closed characteristics are Minkowski billiards in $T_1$ with respect to the norm $\|\cdot\|_{T_2^\circ}$ with minimal $T_2^\circ$-length.
        From the analysis of different billiard trajectories carried out in \cite{AA-O} (cf. \cite{capacity_mahler}), the only possible isotropic gliding in a positive measure set $A$ is when for $t \in A$, $\gamma(t)$ moves along $\partial T_1 \times \partial T_2$.
        Recall that closed billiard trajectories in $
T_1$ of minimal $T_2^\circ$-length correspond to length-minimizing closed polygonal curves with at most $n+1$ vertices that cannot be translated into the interior of $T_1$.
This was proved for the Euclidean case in Theorem 1.1 of~\cite{bezdek-bezdek}, and was later extended to Minkowski billiards in, e.g., \cite[Theorem 2.1]{non_symmetric_mahler} and~\cite[Theorem 1]{Mink-Bill-Rudolf}.
        Moreover, from the proofs ibid it is clear that if one starts from a $T_2^\circ$-length minimizer billiard trajectory $\pi_q\gamma$ which is not a polygonal curve, there is a choice of a polygonal curve $l$ composed out of $n+1$ points along $\pi_q \gamma$ which is a billiard trajectory of the same $T_2^\circ$-length. 
        One can check that $l$ corresponds to an action-minimizing closed characteristic $\gamma_l$ with the same action as $\gamma.$
        Suppose that the vertices of $l$ include the points $\pi_q \gamma(t_1), \pi_q \gamma(t_2).$
        Denote $\Delta q = \pi_q (\gamma(t_2) - \gamma(t_1)),$ and let $p^* \in T^2$ attain the support of $h_{T_2}(\Delta q) = \sup_{p \in T_2} \langle \Delta q, p\rangle = \langle \Delta q, p^* \rangle.$  
        We get
        $$ \int_{t_1}^{t_2} h_{T_2}(\pi_q \dg(t)) dt \geq \int_{t_1}^{t_2} \langle (\pi_q \dg(t), p^* \rangle = \langle \Delta q, p^* \rangle.$$
        By minimality of $\pi_q \gamma$ with respect to $T_2^\circ$-length, the above inequality is an equality, which holds only if 
        $$ p^* \in \arg\max_{p \in T_2} \langle \pi_q \dg(t), p \rangle \; \text{ for almost every }t\in (t_1,t_2).$$
        Hence $\pi_q \dg(t) \in N_{T_2}(p^*)$.
        Similarly, for $\pi_p \gamma$ we have a finite number of time intervals where $\pi_p \dg(t) \in N_{T_1}(q^*)$ for some point $q^*$ depending only on the interval. Hence, for almost every $t \in A$, we can find an intersection of the above intervals for $\pi_q \dg$ and $\pi_p \dg$ so that for some $t \in (a,b) \subset S^1,$ $\pi_q \dg(t) \in N_{T^2}(p^*)$ and $\pi_p \dg(t) \in N_{T^1}(q^*)$.
        Since $\dg(t)$ belongs to an isotropic subspace of $N_K(\gamma(t))$, one has $p^*\perp q^*,$ and hence $\gamma((t_1,t_2))$ moves along the coisotropic subspace $\{x \in \R^{2n} :h_{T_2}(p^*) = \langle \pi_q(x), p^* \rangle\} \cap \{ x\in \R^{2n} : h_{T_1}(q^*) = \langle \pi_p(x), q^* \rangle\}$.
        
    \end{proof}

\begin{remark}\label{polytopeGlidingRemark}
    If $K$ is a convex polytope with outer normals $\{n_1,\ldots,n_k\}$ and $\gamma$ is a characteristic (not necessarily action-minimizer or even closed). It still holds that for almost every $t$ where $\dg(t) = \sum_{j=1}^m a_j J n_{i_j}$ with $\omega(n_{i_j},n_{i_k}) = 0$, there exists a coisotropic face $E$ and an interval $I=(t,t+\delta)$ in which $\gamma(I) \subset E,$ but it is no longer true that one can choose $E$ so that $\dg(I) \subset J N_K(E).$
    Indeed, 
    assume that $\dg(t_0) = \sum_{j=1}^m a_j Jn_{i_j}$
        as there are a finite set of normals, outer semi-continuity of $N_K$ implies that there is $\delta>0$ small enough so that
        $$\bigcup_{s \in [t_0, t_0+\delta)} N_K(\gamma(s)) \subset N_K(\gamma(t_0)).$$        
        Consider the face 
        $$E = \partial K \cap \{x : \langle x, n_{i_j}\rangle = h_K(n_{i_j}), \quad j\in\{1,\ldots,m\}\}.$$
        From the fact that $\omega(n_{i_j}, n_{i_k}) = 0,$ $E$ is coisotropic.
        Let
        $$ \psi_i(t) = h_K(n_i) - \langle n_i, \gamma(t) \rangle \geq 0, \quad i \in \{1,\ldots, k\}.$$
        Consider the function
        $$\psi(t) = \sum_{j=1}^m a_{j}\psi_{i_j}(t).$$
        Since $\dg(s) \in J N_K(\gamma(t_0))$ for almost every $s \in (t_0, t_0 + \delta)$,
        $$ \frac{d \psi(s)}{ds} = - \langle \sum a_{j} n_{i_j}, \dg(s) \rangle =  \langle \dg(t_0), -J \dg(s) \rangle \leq 0.$$
        This follows from the fact that $\dg(t_0) \in T_K(\gamma(t_0)).$
        Hence $\psi \equiv 0$ on $(t_0, t_0 + \delta)$,  but as $a_{i_j} > 0,$ we have $\psi_{i_j} = 0$ for all $j \in \{1,\ldots, m\},$ and hence $\gamma((t_0, t_0 + \delta)) \subset E.$

        \medskip

        However, one can still construct examples for polytopes with isotropic gliding for non action-minimizers.
        Consider a convex polytope $K_1$ in $\R^4$ with an edge $E$ which is the intersection of three hyperplanes with normals $n_1,n_2,n_3$, with $\omega(n_i,n_j) \neq 0, \; i\neq j, \; i,j \in \{1,2,3\},$ such that $TE \subset JN_K(E),$ i.e. the direction of the edge is in the positive span of $Jn_1, Jn_2, Jn_3.$
        This is called a \emph{bad edge} in \cite{Ch-H}.
        Let $K_2 \subset \R^4$ be a different polytope with a Lagrangian face $F$ with $N_K(F) = \text{cone}\{n_4,n_5\}.$
        In the symplectic product $K_1 \times K_2 \subset \R^8,$ consider a point in $\text{relint}(E) \times \text{relint}(F)$.
        For a small interval $I$, choose a fat Cantor set $A \subset I$ with positive measure. There exists a characteristic with velocities along the edge $E$ for $t \in I \setminus A,$ and with nontrivial convex combination of $n_4,n_5$ for $t \in A.$
        We get that $A$ is an isotropic gliding set.

\end{remark}

\medskip\medskip

We proceed to the proof of Theorem~\ref{compactQuotientTopologyThm}. 
Let us first give the definition of the equivalence relation $\sim$.

We denote by $\Omega_L(\gamma) \subset S^1$ the open set of times $t$ for which $\gamma$ \emph{strictly} moves along a coisotropic face at time $t$.
Since $\gamma$ has positive action,
$\Omega_L(\gamma) \neq S^1.$ However, it may happen that $S^1 = \overline{\Omega_L(\gamma)}.$
\begin{definition}\label{coisotropicCollapseDef}
Let $\gamma_1,\gamma_2 \in \text{Sys}(K).$
We say that $\gamma_1$ and $\gamma_2$ are \emph{equal up to coisotropic faces} and write $\gamma_1 \sim \gamma_2$, if:
\begin{enumerate}
    \item $\Omega_L(\gamma_1) = \Omega_L(\gamma_2)$; denote this common set by $\Omega_L$.
    \item For all $t \in S^1 \setminus\Omega_L$, $\gamma_1(t) = \gamma_2(t).$
\end{enumerate}  
\end{definition}

 The condition of strictly moving along each coisotropic face guarantees that equivalent closed characteristics must be equal at those times when they pass from moving along one coisotropic face to another.

\medskip

\begin{definition}\label{pseudoMetricDef}
    
The quotient pseudo-metric between equivalence classes is defined by
\[
d_\sim([\gamma_1], [\gamma_2]) = \inf \sum_{i=1}^k d_{H^1}(z_i,w_i),
\]
where the infimum is taken over finite sequences $z_1,\dots,z_k, w_1,\dots,w_k \in \mathrm{Sys}(K)$ such that $z_1\sim \gamma_1$, $z_{i+1} \sim w_i$ for $i=1,\dots,k-1$, and $w_k \sim \gamma_2$.

\end{definition}

\medskip

We start with the following~lemma.

    \begin{lemma}\label{c0Impliesc1Lemma}
        Let $\gamma_k \in W^{1,2}([0,T],\partial K)$ be a sequence of characteristics converging $C^0$ to $\gamma$. Let $X \subset [0,T]$ be a compact set of positive measure so that for $t\in X$, $\dg(t)$ is on an extreme ray of $N_K(\gamma(t))$.
        Then
        $$\int_{X} \|\dg(t) - \dg_k(t) \|^2 dt \xrightarrow{k \to \infty} 0.$$  
    \end{lemma}

    \begin{proof}
    
    We normalize $K$ with $\ehzcap(K) = 1,$ and normalize $\gamma$ with $h_K(-J\dg(t)) = 2$ for almost every $t \in X.$
    For $x \in \partial K$ set $C(x) = J (N_K(x) \cap S),$
    where $S = \{ h_K(v)=2\}.$
    Note that $x \mapsto C(x)$ is outer semi-continuous, and $t \mapsto \dg(t) \in C(\gamma(t))$ is a measurable selection of an extreme point in $C(\gamma(t))$.
    
    The idea of the proof is to track a ``support vector drift" of $\gamma_k(t) - \gamma(t).$ Roughly speaking, since $\dg(t)$ is an extreme element of $C(\gamma(t))$ and $\dg_k(t)$ is very close to $C(\gamma(t))$ from outer semi-continuity, there exists a support vector $u(t)$ to $C(\gamma(t))$ such that one has that $\langle \dg_k(t)-\dg(t), u(t) \rangle$ is never positive, and whenever $\dg_k(t)$ is far enough from $\dg(t)$ it is negative and bounded away from $0$. In a tiny interval where $u(t)$ is nearly constant, the derivative of $\langle \gamma_k(t)-\gamma(t), u \rangle$ is the negative value $\langle \dg_k(t)-\dg(t), u(t) \rangle$, which forces $\gamma_k$ to move away from $\gamma$ in direction $-u$ in a way independent of $k$. This forces $\gamma_k$ to be $C^0$-far from $\gamma$ which yields a contradiction.

    \medskip

    More precisely, Straszewicz’s theorem states that exposed points are dense in the set of extreme points. Hence, for $\delta_0 > 0$ small we can find a measurable selection $u(t) \in \R^{2n}$ with $\|u(t)\| = 1$ and $\varepsilon_0, \varepsilon_2$ such that for almost every $t \in X$,
    $$ \langle n-\dg(t),u(t) \rangle < \varepsilon_2, \forall n \in C(\gamma(t)), \text{ and } \langle n-\dg(t), u(t) \rangle < -\varepsilon_0 \text{ whenever } \|n - \dg(t)\| \geq \delta_0 .$$
    Moreover, one can choose $\varepsilon_0, \varepsilon_2$ to satisfy $\frac{\varepsilon_2}{\varepsilon_0} \to 0$ as $\delta_0 \to 0$.
    The uniformity comes from the compactness of the set $\{ (t,n) \in X \times C(\gamma(t)) : \|n - \dg(t)\| \geq \delta_0 \}.$

    We will prove the lemma by showing that on a full measure subset of $X,$ $\lim_{k\to \infty}\|\dg(t) - \dg_k(t)\| \leq c ( \frac{\epsilon_2}{\varepsilon_0} + \delta_0),$ where $c>0$ is a constant. Taking the intersections of these subsets over smaller and smaller $\delta_0$ shows that the limit vanishes for almost every $t \in X.$ 
    In order to do so, it is enough to prove that for a sufficiently small interval $I'$ surrounding a Lebesgue point of $X$, 
    $$\int_{I'} \|\dg(t)-\dg_k(t)\|^2 dt < c (\frac{\varepsilon_2}{\varepsilon_0}+\delta_0) |I'| + f(|I'|) +  g(k, |I'|),$$
    where $c>0$ is a constant, $f(|I'|) = o(|I'|)$ as $|I'| \to 0,$ and $g(k, |I'|) = o(1)$ as $k \to \infty$.

\medskip

    Let $I'$ be such an interval with $|X \cap I'| \geq |I'| - \varepsilon/2,$ where $\varepsilon = o(|I'|).$
    Lusin's theorem gives a compact set $I \subset I'$ of measure $|I| > |I'|-\varepsilon$ where for almost every $t \in I$, $t \mapsto \dg(t)$ and $t \mapsto u(t)$ are uniformly continuous and $\dg(t)$ is on an extreme ray of $N_K(\gamma(t)).$
    Let $u'$ be a $C^1$ uniform approximation of $u$ with $\|u(t) - u'(t)\| \leq \rho$ for every $t \in I$, and $\|\dot{u}'(t)\| \leq C$ depending on $\rho$ for every $t \in I'$.
    From outer semi-continuity, for $k$ large enough, there is a uniform distance $d_{C^0}(\dg_k(t), C(\gamma(t))) < \min\{ \delta_0, \varepsilon_2\}$ for every $t \in I'$. In addition $\| \gamma_k(t) - \gamma(t)\| < \varepsilon_1 \ll \frac{\varepsilon_0}{C}.$
    Define the ``bad set"
    $$A = \{ t \in I: \|\dg(t) - \dg_k(t)\| > 2 \delta_0 \}.$$
    Since for $\dg_k(t)$ there exists $n(t) \in C(\gamma(t)),$ with $\|n(t) - \dg_k(t)\| < \min\{\delta_0, \varepsilon_2\},$ we have that for $t \in A,$
    $$ \| n(t) - \dg(t)\| \geq 2\delta_0 - \delta_0 = \delta_0.$$
    Hence
    $$ \langle \dg_k(t) - \dg(t),u(t) \rangle < -\varepsilon_0 + \|\dg_k(t)-n(t)\| \|u(t)\| < -\frac{\varepsilon_0}{2}.$$
    Define the ``drift" functional $\Phi(t) = \langle \gamma_k(t) - \gamma(t), u'(t) \rangle$ and note that $ | \Phi(t) | <  \varepsilon_1 (1+\rho).$
    We have
    $$ \dot{\Phi}(t) = \langle \dg_k(t) - \dg(t), u'(t) \rangle + \langle \gamma_k(t) - \gamma(t), \dot{u}'(t) \rangle < \langle \dg_k(t) - \dg(t), u'(t) \rangle + C\varepsilon_1.$$
    We also note that $\|\dg\|, \|\dg_k\| < C_1$ are uniformly bounded by the normalization $h_K(-J\dg) = h_K(-J\dg_k) = 2.$
    For $t \in A,$ we have 
    $$\langle \dg_k(t) - \dg(t), u'(t) \rangle <  -\frac{\varepsilon_0}{2} + 2C_1 \rho .$$
    For $t \in I \setminus A$ we have
    \begin{align*}
    \langle \dg_k(t) - \dg(t), u'(t) \rangle &= \langle \dg_k(t) - n(t), u(t) \rangle + \langle n(t) - \dg(t), u(t) \rangle + \langle\dg_k(t) - \dg(t), u'(t)-u(t)\rangle \\
    &< 2\varepsilon_2 + 2C_1 \rho .
    \end{align*}
    For $t \not\in I$ we still have a uniform bound
    $\dot{\Phi}(t) < C_2 + C \varepsilon_1.$
    Integrate $\dot{\Phi}$ over $I' = (t_1,t_2)$ to get
        $$ \Phi(t_2) - \Phi(t_1) = \int_{I'\setminus I} \dot{\Phi}(t)dt + \int_{I\setminus A} \dot{\Phi}(t) dt + \int_A \dot{\Phi}(t) dt.$$
The above bounds on $\dot{\Phi}$ and on $\Phi$ now readily yield
    $$ -2\varepsilon_1 (1+\rho) < (C_2+C\varepsilon_1) \varepsilon + (2 \varepsilon_2 + 2C_1 \rho) (|I| - |A|) + (-\frac{\varepsilon_0}{2} + 2C_1 \rho + C \varepsilon_1)|A| $$
    which bounds the measure of $A$ by
    $$ |A| < \frac{2 \varepsilon_1 (1+\rho) + \varepsilon (C_2+C \varepsilon_1)  + (2C_1 \rho+ 2\varepsilon_2 )|I'|}{\frac{\varepsilon_0}{2} + 2\varepsilon_2 - C\varepsilon_1}.$$
    We note that $|\{ t \in I' : \|\dg_k(t) - \dg(t)\| \geq 2\delta_0\}| < |A| + \varepsilon$,
    Finally, from the uniform bound $\|\dg_k(t)\|,\|\dg(t)\| \leq C_1,$ one has
    $$ \int_{I'} \| \dg(t) - \dg_k(t)\| dt < 2 \delta_0 |I'| + 2 C_1 (|A| + \varepsilon).$$
    This proves the required bound by sending first $\varepsilon_1$, then $\rho$, and then $\varepsilon$ to zero.
    
    \end{proof}

\begin{lemma}\label{coisotropicFaceMeansStrictlyLemma}
    Let $K \subset \R^{2n}$ be convex, and suppose that $\gamma \in \text{Sys}(K)$ moves along a coisotropic face $E$ at time $t$ for $t \in (a,b)$, and $\dg(t)$ is not a constant vector which lies on an extreme ray of $N_K(E)$. Then there exists a perturbation $\gamma' \in \text{Sys}(K)$ arbitrarily $H^1$-close to $\gamma$ that strictly moves inside a coisotropic face $E'$ at time $t$ for $t \in (a,b).$
    Moreover, one can choose $\gamma'$ and $E'$, so that $\gamma'(s) + \varepsilon(s) J n \in E'$ for every $s \in (a,b),$ $n \in N_K(E')$ and $\varepsilon(s) > 0$ sufficiently small.
    In particular, one can choose $\dg'(s)$ to never lie on extreme rays of $N_K(E').$
\end{lemma}
\begin{proof}
    Suppose that $\gamma \in \text{Sys}(K)$ moves along a coisotropic face $E$ at time $t$ for every $t \in (a,b)$.
    Write $\dg(t) = \sum_{i=1}^m a_i(t) J n_i$ for $\{n_1,\ldots, n_m\}$ extreme points of $N_K(E),$ and such that for every $i$ there exists $t \in (a,b)$ with $a_i(t) > 0.$
    We can assume without loss of generality that 
    $$E= \left\{x \in \partial K : h_K(n_i) = \langle x, n_i \rangle, \; \forall i\in \{1,\ldots,m\} \right\}. $$
    If for some $t$, $\gamma(t) \in \text{relint}(E)$, then the straight segments $l_1 = [\gamma(a), \gamma(t)]$ and $l_2 =[\gamma(t), \gamma(b)]$ both lie inside the relative interior, and one can, using a convex combination, connect $l_1\#l_2$ with $\gamma$ to get a small perturbation $\gamma'$ with $\gamma'([a,b]) \subset \text{relint}(E).$
    Hence we can assume $\gamma([a,b]) \subset F$ for some face $F$ of $E$.
    Indeed, if $\gamma(t_1)$ and $\gamma(t_2)$ do not lie on the same face, then again a convex combination between $\gamma$ and the straight segment $[\gamma(t_1), \gamma(t_2)]$ will put some part of $\gamma$ in the interior. 

    It is a useful fact that $\gamma$ remains an action minimizer after reordering the velocities on $(a,b)$. 
    Indeed, let $\beta_i = \int_a^b a_i(t)dt.$ Then for every $\phi: [a,b] \to (0,\infty)^m$ measurable with $\int_a^b \phi(t)dt = (\beta_1,\ldots,\beta_m)$, define $\gamma^\phi$ by
    \[
\gamma^\phi(s) := \gamma(a) + \int_{a}^s \phi(r)\,dr,
\quad s \in (a,b).
\]
Then one can check that, up to time reparamterization, $\gamma^\phi$ is an extremizer of the Clarke functional. This is due to the fact that $\dg(t) \in J N_K(E)$ for $t \in (a,b).$ 
By Clarke's duality, 
for every $\phi$ the curve $\gamma^\phi$ is a translated action-minimizing closed characteristic on $\partial K$, moving along $E$.
This follows from the fact that $E$ is composed of the supporting hyperplanes with normals that participate in $\dg(t)$ for some $t \in (a,b).$

    If for every normal $\nu \in N_K(F)$ and every $i$, $\langle \nu, J n_i \rangle \leq 0,$ then in particular $D_K(\gamma(t)) \subset TE,$ and, maybe after a small perturbation which guarantees that $\gamma((a,b)) \subset \text{relint}(F)$, a sufficiently small translation in direction $J n_i$ remains inside $E$.
    In this situation, we set $\gamma'=\gamma$ and we are done.

    Set $\gamma_i$ to be a reordering of $\dg$ so that $\dg_i$ is equal to $J n_i$ for a very small time interval near $a$, and then follow arbitrarily close to $\dg$.
    Let $v_i = \gamma_i(a) - \gamma(a)$ be the translation, and note that $v_i \in TE.$
    Since $J n_i \in D_K(\gamma_i(a)),$ we have that for every $\nu \in N_K(\gamma_i(a))$, $\langle \nu, Jn_i \rangle \leq 0.$
    This is the case after an arbitrarily small translation in the direction $v_i$. In addition, note that as $\dg(s) = \dg_i(s)$ for $s \not\in(a,b)$, we have that $v_i \in T_{\gamma(s)} \partial K$ for every $s$, and $v_i \perp -J\dg(s)$ for $s \not\in (a,b)$. 
    Hence any convex combination of $\gamma$ and $\gamma_i$ belongs to $\text{Sys}(K)$, and one can choose it to lie  arbitrarily close to $\gamma$. 

    Repeating this process for every $i=1,\ldots,m$, we get the required.

\end{proof}

    \begin{proof}[Proof of Theorem~\ref{compactQuotientTopologyThm}]
    Let $\gamma_k \in \text{Sys}(K)$ be a sequence of action-minimizing closed characteristics normalized with $h_K(\dg(t))=2.$
    From \cite[Proposition~2.1]{Matijevic}, $\gamma_k$ has a converging subsequence in the $C^0$ topology, also denoted $\gamma_k$.
    Denote its limit by $\gamma$.
    Denote by $\Omega \subset S^1$ to be the set of times where $\gamma$ moves along a coisotropic face. 
    For $t \in \Omega(\gamma),$ let $E(t)$ be the maximal (with respect to inclusion) coisotropic face of $\partial K$, where $\gamma$ moves along $E=E(t)$ at time $t$. By definition of moving along $E$, there is an interval surrounding $t$ where $\gamma \subset E$ and $\dg \subset N_K(E)$. Denote by $(a(t),b(t))$ the maximal such interval.
    One can cover $\Omega$ up to measure $\frac{1}{k}$ using a finite set $\mathcal{I}_k$ of such intervals.
    We remove from $\mathcal{I}_k$ those intervals where $\dg(t)$ is constant and lies on an extreme ray of $J N_K(E(t)).$
    For each interval $(a,b) \in \mathcal{I}_k$, using Lemma~\ref{coisotropicFaceMeansStrictlyLemma}, we sequentially approximate $\gamma$ to get a $\gamma'_k \in \text{Sys}(K)$ that satisfies the conditions in Lemma~\ref{coisotropicFaceMeansStrictlyLemma} for every interval in the cover.
    Our goal now is to find $\tilde{\gamma}_k$ that would satisfy $\tilde{\gamma}_k \sim \gamma'_k$ and $d_{H^1}(\gamma_k, \tilde{\gamma}_k) \ll 1.$
    Let $E(t)$ now denote the corresponding coisotropic face of $\gamma'_k$ for whenever $t$ belongs to an interval $(a,b) \in \mathcal{I}_k.$
        
    Define the measurable function $v_k(t) = \text{proj}_{JN_K(E(t))}(\dg_k(t))$. 
    Let $\tilde{\gamma}_k$ be defined as follows. For $t \notin \bigcup_{I \in \mathcal{I}_k} I, \quad \tilde{\gamma}_k(t) = \gamma'_k(t)$. For $t \in (a,b)$ we set 
        $$\tilde{\gamma}_k(t) = \gamma'_k(a) + \int_a^t v_k(s) -\zeta ds ,$$
        where 
        $$\zeta = \frac{\int_a^b v_k(s) ds - (\gamma'_k(b) - \gamma'_k(a))}{b-a}.$$

    We note that, maybe after a small adjustment of $\gamma'_k$ to not follow extreme points of $JN_K(E),$ one has for $k$ sufficiently large, that $\tilde{\gamma}_k \sim \gamma'_k$. 
    From outer semi-continuity of $N_K(x)$ together with compactness of $S^1$, $v_k(t)$ is uniformly arbitrarily close to $\dg_k(t)$ for sufficiently large $k$. On the other hand, from $C^0$ convergence, we know that $\gamma_k(a), \gamma_k(b)$ are arbitrarily close to $\gamma'_k(a), \gamma'_k(b)$ respectively. We get that  $\tilde{\gamma}_k$ is $H^1$-close to $\gamma_k$ for almost every $t \in \bigcup_{I \in \mathcal{I}_k} I$. 
    Lemma~\ref{c0Impliesc1Lemma} then readily implies that $$\int_{S^1 \setminus \Omega(\gamma)} \|\dot{\tilde{\gamma}}_k(t) - \dg_k(t) \|^2 dt \xrightarrow{k \to \infty} 0.$$ 
    Hence, since $|\Omega| - |\bigcup_{I \in \mathcal{I}_k} I|$ was chosen arbitrarily small, we have $d_{H^1}(\gamma_k, \tilde{\gamma}_k) \to 0,$ and the fact that $\tilde{\gamma}_k$ is equal up to coisotropic faces to $\gamma'_k$, which is arbitrarily $H^1$-close to $\gamma$, proves the required.

    \end{proof}

\section{Dynamical characterization of Cuts Additivity}\label{dynCutSection}

    The goal of this section is to prove Proposition~\ref{additiveDescriptionThm} and provide a combinatorial description for the case of convex polytopes.

\begin{proof}[Proof of Proposition~\ref{additiveDescriptionThm}]

    Suppose that $\gamma$ is an action minimizing closed characteristic that hits $H$ at $t_1,t_2 \in S^1$ and $\gamma(t_2) - \gamma(t_1) \in \R_+ J n$. 
    Let $\gamma_1 = \gamma([t_1,t_2]) \# [\gamma(t_2),\gamma(t_1)],$ and $\gamma_2 = \gamma([t_2,t_1]) \# [\gamma(t_1),\gamma(t_2)].$ We note that $\gamma_1$ and $\gamma_2$ are closed characteristics on the boundary of $\partial K_1$ and $\partial K_2$ respectively, and hence $\mathcal{A}(\gamma_i) \geq \ehzcap(K_i), \; i=1,2$.
    On the other hand we have $\mathcal{A}(\gamma_1) + \mathcal{A}(\gamma_2) = \mathcal{A}(\gamma) = \ehzcap(K).$ using Theorem~1.8 from \cite{pazit} (subadditivity for hyerplane cuts), we have $\ehzcap(K) \leq \ehzcap(K_1) + \ehzcap(K_2),$ and hence $\gamma_1$ and $\gamma_2$ must be action-minimizers, and $\ehzcap(K) = \ehzcap(K_1) + \ehzcap(K_2).$

    In the other direction, let $H$ be a hyperplane cut of $K$ into $K_1$ and $K_2$ with unit normal $n$, such that $\ehzcap(K) = \ehzcap(K_1) + \ehzcap(K_2).$ Assume without loss of generality that $0 \in H.$
    Let $\gamma_1$ and $\gamma_2$ be action one minimizers of $\Phi_{K_1}, \Phi_{K_2}$ respectively. 
    Since $\ehzcap(K_i) < \ehzcap(K),$ we know that both $\gamma_1$ and $\gamma_2$ have parts $l_1, l_2 \subset S^1$ with velocity $J n, -Jn$ respectively.
    Following \cite{pazit}, one can choose $\gamma_1$ and $\gamma_2$ so that $l_1$ and $l_2$ are connected. 
    Denote $a_1 = |\gamma_1(l_1)|,a_2 =|\gamma_2(l_2)|$ their respective lengths.
    Let $\tilde{\gamma}_1 = a_2 \gamma_1,$ $\tilde{\gamma}_2 = a_1 \gamma_2 + x_0,$ where $x_0$ is chosen so that $\tilde{\gamma}_2(l_2)$ is glued to $\tilde{\gamma}_1(l_1).$
    Define $\gamma = \tilde{\gamma}_1 \# \tilde{\gamma}_2$ the concatenation of $\tilde{\gamma}_1$ and $\tilde{\gamma}_2$ along $\tilde{\gamma}_1(l_1),$ reparametrized to have period 1.
    Note that since $0 \in H$, one has $h_{K_i}(\dg_i(t))|_{l_i} = 0,$ for $i=1,2.$
    On the other hand for $t \not \in l_i$, $h_{K_i}(\dg_i(t)) = h_K(\dg_i(t)).$
    We thus have $a_2 \Phi_{K_1}(\gamma_1) + a_1\Phi_{K_2}(\gamma_2) = \Phi_K(\gamma)$, and hence

    $$a_2 \ehzcap(K_1)^{1/2} + a_1 \ehzcap(K_2)^{1/2} = \Phi_K(\gamma) \geq \left(\mathcal{A}(\gamma) \ehzcap(K) \right)^{1/2} = ( (a_1^2 + a_2^2) \ehzcap(K))^{1/2}. $$
    After squaring both sides and rearranging one gets
    $$ -(a_1 \ehzcap(K_1)^{1/2} - a_2 \ehzcap(K_2)^{1/2})^2 \geq 0.$$
    This is only possible if $a_1 \ehzcap(K_1)^{1/2} = a_2 \ehzcap(K_2)^{1/2}$, and if the above inequalities are equalities. In particular, the lengths of the segments of the corresponding closed characteristics on $\partial K_1$ and $\partial K_2$ on the hyperplane $H$ agree, and $\frac{\gamma}{\mathcal{A}(\gamma)}$ is a minimizer of $\Phi_K$ and thus corresponds to a closed characteristic on $\partial K$. 
    In particular, we get that part of this closed characteristic is a translation $x_0 \in \R^{2n}$ of the part of the closed characteristic $z_1$ corresponding to $\gamma_1$ on $\partial K_1$ which lies outside $H$.
    From the fact $u \in N_K(x_1) \cap N_K(x_2) \implies \langle u, x_2 - x_1 \rangle = 0,$ we have that $x_0 \perp -J \dot{z}_1$ everywhere on $\partial K$, and one can integrate over $z_1 \cap \partial K$ to get $x_0 \perp -J \int \dot{z}_1(t) dt = Cn$, for some $C>0$. Hence  $\frac{\gamma}{\mathcal{A}(\gamma)}$ indeed splits along $H$ to closed characteristics on $\partial K_1$ and $\partial K_2$ respectively.
    
\end{proof}

We now discuss a combinatorial description of Proposition~\ref{additiveDescriptionThm} for convex polytopes.
For this let us recall the formula from \cite[Theorem~1.1]{pazit}:
Let $K$ be a convex polytope. Then
\begin{equation}\label{formulaEq} \ehzcap(K) = \frac{1}{2} \left[ \max_{(\beta_i, n_i)_{i=1}^k \in M(K)} \sum_{1\leq j < i < k} \beta_i \beta_j \omega(n_i,n_j)  \right]^{-1},\end{equation}
where 
$$ M(K) = \left\{ (\beta_i, n_i)_{i=1}^k : n_i \text{ unit outer normal to } \partial K, \beta_i\geq0, \sum \beta_i h_K(n_i) = 1, \sum \beta_i n_i = 0 \right\} .$$
Here $M(K)$ includes all finite sequences and we allow for repetitions.
We note that \cite[Lemma~3.5]{pazit} proves that there exists a maximizing sequence for which each unit outer normal of $\partial K$ appears at most once.

\begin{definition}\label{combinatorialCutDef}
    Let a convex polytope $K$ be cut by a hyperplane $H$ with normal vector $v$ into $K_1, K_2$, at height $h_K(v) - t$, for $t \in (0, h_K(v))$.
     A maximizing sequence $(\beta_i, n_i)_{i=1}^k \in M(K)$ of \eqref{formulaEq} is a combinatorial cut for $(v,t)$, if there exists $1 < m < k$ satisfying that $n_1,\ldots,n_m$ is a sequence of unit outer normals of $K_1$, and $\sum_{i=1}^m \beta_i n_i = c v $, with 
     \begin{equation} \label{cEq}
     c = \frac{\sum_{i=1}^m \beta_i h_K(n_i) - 2 \ehzcap(K) \sum_{1\leq i < j \leq m} \beta_i \beta_j \omega(n_i,n_j)}{h_K(v) - t}.
     \end{equation}
\end{definition}

\begin{remark}
    The sequence $(\beta_i,n_i)_{i=1}^m$ together with the normal $-v$, after appropriate scaling, comprises a maximizing sequence in \eqref{formulaEq} for $K_1,$ while $(\beta_i,n_i)_{i=m+1}^k$ together with the normal $v$ does the same for $K_2$.
\end{remark}
\begin{remark}\label{arbitraryCutRemark}
    Let us analyze the constraint \eqref{cEq}.
    First note that one always has $\sum_{i=1}^m \beta_i h_K(n_i) \geq c h_K(v),$ with equality if and only if there exists $x \in \partial K,$ with $n_1,\ldots,n_m,v \in N_K(x).$
    Suppose that there exists such an $x$.
    Then if one is able to find an element $(\beta_i,n_i) \in M(K)$ satisfying $\sum_{i=1}^m \beta_i n_i = c v $ for some $c>0$, and 
    $$0 < \sum_{1\leq i < j \leq m} \beta_i \beta_j \omega(n_i,n_j) < \frac{c h_K(v)}{2\ehzcap(K)},$$ 
    then there exists $0 < t < h_K(v)$ so that $(\beta_i,n_i)$ is a combinatorial cut for $(v,t).$
    Moreover, if one is able to multiply the coefficients $(\beta_i)_{i=1}^m$ by a small constant $\lambda$ and still find a corresponding maximizing element in $M(K)$, then there exists a combinatorial cut for every small enough $t$. 
    
\end{remark}

\begin{lemma}\label{combinatorialCutLemma}
    A convex polytope $K$ is additive with respect to a hyperplane cut with normal $v$ at height $h_K(v)-t$ if and only if there exists a combinatorial cut for $(v,t)$. Moreover, consider the set $B$ of closed characteristics $\gamma$ which are cut by $H$ with $\gamma(0), \gamma(T) \in H$ and $\gamma(T) - \gamma(0)$ parallel to $J v$. 
    Then $B$ is homotopically equivalent to the space of combinatorial cuts.
    
\end{lemma}
\begin{proof}
    Let $$(\beta_i, n_i)_{i=1}^k \in M(K)$$ be a combinatorial cut for $(v,t).$
    Denote $$A_1 = \sum_{1\leq i < j \leq m} \beta_i \beta_j \omega(n_i,n_j),$$ and $$A_2 = \sum_{m+1\leq i < j \leq k} \beta_i \beta_j \omega(n_i,n_j).$$ Note that $$A_1 + A_2 = \sum_{1\leq i < j \leq k} \beta_i \beta_j \omega(n_i,n_j) =: A$$
    Indeed, from the constraint $\sum_{i=1}^k \beta_i n_i = 0$ we have $\sum_{i=m+1}^k \beta_i n_i = -c v$, and 
    $$ A = A_1 + A_2 + \sum_{i=1}^m \sum_{j=m+1}^k \beta_i \beta_j \omega(n_i,n_j) = A_1 + A_2 + \omega(cv,-cv) = A_1 + A_2.$$
    Denote $$H_1 = \sum_{i=1}^m \beta_i h_K(n_i) - c (h_K(v) - t).$$
    Consider the sequence $$(\beta'_i, n'_i)_{i=1}^{m+1} = (\frac{\beta_i}{H_1},n_i)_{i=1}^m \cup (\frac{c}{H_1}, -v).$$
    We have $$\sum_{i=1}^{m+1} \beta'_i n'_i = \frac{1}{H_1}\sum_{i=1}^m \beta_i n_i - \frac{c}{H_1}v = 0.$$
    Moreover, 
    $$\sum_{i=1}^{m+1} \beta'_i h_{K_1}(n'_i) = \frac{1}{H_1}\sum_{i=1}^m \beta_i h_K(n_i) - \frac{c}{H_1} (h_K(v) - t) = 1.$$
    Since the constraints of $M(K_1)$ are satisified, we have
    $\ehzcap(K_1) \leq \frac{H_1^2}{2A_1}.$
    Consider the sequence of normals for $K_2$:
    $$(\beta''_i, n''_i)_{i=m}^k = (\frac{c}{H_2}, v) \cup (\frac{\beta_i}{H_2}, n_i),$$
    where $$H_2 = \sum_{i=m+1}^k \beta_i h_K(n_i) + c (h_K(v) - t).$$
    The constraints for $M(K_2)$ are again satisfied and
    $$ \ehzcap(K_2) \leq \frac{H_2^2}{2A_2} = \frac{(1 - H_1)^2}{2A_2}.$$
    Denote $H =\sum_{i=1}^m \beta_i h_K(n_i)$ and note that after substituting $c$ using~\eqref{cEq} one gets 
    $$H_1 = H - (H - 2\ehzcap(K) A_1) = 2 \ehzcap(K) A_1 = \frac{A_1}{A},$$
    $$H_2 = 1-H_1 = \frac{A_2}{A}.$$
    Finally,
    $$ \ehzcap(K_1) + \ehzcap(K_2) \leq \frac{A_1}{2A^2} + \frac{A_2}{2A^2} = \frac{1}{2A} = \ehzcap(K), $$
    and from subadditivity we get $\ehzcap(K_1) + \ehzcap(K_2) = \ehzcap(K).$

    In the other direction, if $\ehzcap(K_1) + \ehzcap(K_2) = \ehzcap(K)$, then applying Formula~\eqref{formulaEq} to $K_1$ and $K_2$ separately with each normal appearing at most once, we get two sequences that one can show, after repeating similar calculations as above, that they form a combinatorial cut.

    \medskip

    Let $\gamma \in B$, and let $\gamma_1,\gamma_2$ be the two pieces obtained by cutting $\gamma$ along $H$, so that $\gamma_i$ is an action-minimizing closed characteristic on $\partial K_i$, $i=1,2$.

    By Lemma~\ref{nonExtremeIsIsotropicLemma}, whenever $\dot{\gamma}(t)$ does not lie in an isotropic subspace of $N_K(\gamma(t))$, it is equal to $J$ times a unit outer normal of $\partial K$. Moreover, for any motion of $\gamma$ along a coisotropic face (generated by unit outer normals whose pairwise $\omega$-values vanish), the total movement along the face
    yields a linear combination of the vectors $J n$, where $n$ ranges over the unit outer normals of that face.
    
    All such linear combinations arising from motions along coisotropic faces together with the coefficients of each unit outer normal, yields a representation of $c Jv$ as a linear combination of the $J$-images of unit outer normals. 
    Similarly, $\gamma_2$ gives a representation of $-c Jv$.

    The homotopy equivalence is constructed by choosing for each coisotropic face a fixed order of the normals comprising that face. 
    If there are several ways to write the total movement along each face as a linear combination, we choose, for each face, a continuous function from the normal cone to the coefficients space of convex combinations (there are several standard ways to achieve this).
    Changing the velocities of an action-minimizing closed characteristic while it moves inside a fixed coisotropic face is a convex operation, so we can continuously deform $\gamma_1$ within each coisotropic face, in a canonical way, so that the motion splits into separate segments with velocities $J n$, one for each normal appearing in the unique linear combination, with the specified order of that coisotropic face.
    By means of Lemma~\ref{finiteFacesLemma}, one indeed gets a finite linear combination.
    An analogous procedure for $\gamma_2$ produces a similar homotopy, and performing these homotopies simultaneously for $\gamma_1$ and $\gamma_2$, and then for all $\gamma \in B$ gives a deformation retraction of $B$ onto the subspace of combinatorial cuts given by the specified order of normals for each coisotropic face.

    In addition, the space of combinatorial cuts also deformation retracts onto the same subspace by swapping adjacent normals with vanishing $\omega$ value.

\end{proof}

\section{Cuts Additivity and Zoll}\label{additiveZollSection}

In this section we prove Theorem~\ref{ZollIsAdditiveWellDefinedThm} and Theorem~\ref{additiveIsZollNonSmoothThm}.

\begin{lemma}
\label{additive_or_double_lemma}
Let a convex $K \subset \R^{2n}$ be such that it is foliated by action minimizing closed characteristics and the characteristic direction is well-defined. Then for every hyperplane cut into $K_1$ and $K_2$ one has either 
$$ \ehzcap(K) = \ehzcap(K_1) = \ehzcap(K_2), \; \text{ or } \; \ehzcap(K) = \ehzcap(K_1) + \ehzcap(K_2).$$
\end{lemma}

\begin{proof} [Proof of Theorem \ref{ZollIsAdditiveWellDefinedThm}]

Let a convex body $K\subset \R^{2n}$ be generalized Zoll such that the characteristic dynamics is well defined.
By \cite[Theorem 1.7]{Matijevic}, $K$ is foliated by action-minimizers.
From Lemma \ref{additive_or_double_lemma} we know that for every hyperplane cut one has either 
$$\ehzcap(K_1) + \ehzcap(K_2) = \ehzcap(K) \; \text{ or } \; \ehzcap(K_1) + \ehzcap(K_2) = 2 \ehzcap(K).$$
From the continuity of the EHZ capacity it must be the same value for all hyperplane cuts of $K$. 
It cannot be $2 \ehzcap(K)$ since one can choose a hyperplane so that $K_1$ has arbitrarily small volume and since $\ehzcap(K_1) = \ehzcap(K)$, $K_1$ has arbitrarily large systolic ratio which is a contradiction.
\end{proof}

\begin{proof}[Proof of Lemma \ref{additive_or_double_lemma}]
If both $K_1$ and $K_2$ do not have a closed characteristic with minimal action that intersects the cutting hyperplane $H$ then one has 
$$ \ehzcap(K) = \ehzcap(K_1) = \ehzcap(K_2). $$
Otherwise assume without loss of generality that $K_1$ has an action minimizing closed characteristic $\gamma$ which visits $H$.
Assume $\gamma$ is normalized such that $\| \dg(t) \| = 1$ for all $t$. 
Denote the normal to $H$ into $K_2$ by $n$.
Following~\cite{pazit} we may assume that there is only one segment in the set 
$ \{ t : \dg(t) = Jn \}. $
Denote the period of $\gamma$ to be $T_1$.
Assume without loss of generality that the endpoints of this segment are $\gamma(t_1)$ and $\gamma(T_1)$.
Consider the characteristic in $K$ $\gamma(t)$ for $t \in (0,t_1)$. Since $K$ is foliated by action minimizers and the dynamics is well defined, $\gamma$ has a completion to a closed characteristic $\eta$ on $K$ with minimal action.
Denote the period of $\eta$ by $T_2$.
We wish to prove that $\eta(t) \in \partial K_2$ for every $t \in (t_1,T_2)$.
Otherwise there are times $t_1 \leq t_2 < t_3 \leq T_2$ such that $\eta(t) \in \partial K_1 \setminus \partial K_2$ for every $t \in (t_2,t_3)$. In particular, we can choose $t_2$ so that the derivative $\deta(t_2)$ is defined and satisfies $\langle \deta(t_2), n \rangle < 0.$
Consider the loop $\zeta$ defined up to translations as follows:
\[ 
\dzeta(t) = \left\{ \begin{matrix} 
\deta(t + t_1), & &  t \in (0, t_2-t_1) \\ 
\deta(t + t_1 - t_2) , & & t \in (t_2-t_1,t_2) \\
\deta(t), & & t \in (t_2, T_2) \\
 \end{matrix} \right. 
\]
The action difference is 
\begin{align*}
& \int_0^{T_2} \int_0^t \langle -J \deta(t), \deta(s) \rangle ds dt  - \int_0^{T_2} \int_0^t \langle -J \dzeta(t), \dzeta(s) \rangle ds dt \\
= & 2\int_{t_1}^{t_2} \int_0^{t_1} \langle -J \deta(t), \deta(s) \rangle ds dt = 2  \langle -J (\eta(t_2) - \eta(t_1)) , c Jn \rangle \\
= & -2c \langle \eta(t_2) - \eta(t_1) ,  n \rangle = 0.
\end{align*} 
Where $c >0$ is some constant, and the last equality is due to the fact that $\eta(t_1), \eta(t_2) \in H$.
On the other hand $I_K(\eta) = I_K(\zeta)$, so by Clarke's dual action principle, some constant times $\zeta$ (after a translation) is indeed a closed characteristic with minimal action on $K$.
Since $\dzeta(t) = \deta(t)$ for $t \in (0,t_1)$, we have that $\zeta([t_2-t_1,t_2])$ is a translation of $\eta([0,t_1])$ with some vector $x_0.$ In particular $ \zeta(t_2-t_1) - \zeta(t_2) = c Jn$ for some $c > 0,$ and we have that for in $\zeta(t_2)$, $Jn$ is either tangent to $\partial K$ or points towards the interior of $K$. In any case we have that for all vectors $v \in N_K(\zeta(t_2)),$ $\langle Jn, v \rangle \leq 0.$ In particular, $\langle \partial_+\zeta(t_2), n\rangle \geq 0.$ On the other hand $\partial_+\zeta(t_2) = \deta(t_2)$, and $\langle \deta(t_2), n\rangle < 0.$ A contradiction.

Now we can define a closed characteristic on $K_2$ by taking $\eta$ between $t = t_1$ and $t = T_2$ and connecting $\eta(t_1)$ and $\eta(T_2)$ in a straight line. 
\[ 
\gamma_2(t) = \left\{ \begin{matrix} 
\frac{t}{t_1} \eta(t_1) + \frac{(t_1-t)}{t_1} \eta(T_2) , & & t \in (0,t_1) \\
\eta(t), & & t \in (t_1,T_2) 
 \end{matrix} \right. 
\]
This is a closed characteristic since $\gamma_2(t_1) - \gamma_2(0) = -c Jn .$
Since the concatenation of $\gamma_2$ and $\gamma$ is $\eta$ we have that the action of $\gamma_2$ 
$$ \mathcal{A}(\gamma_2) = \ehzcap(K) - \ehzcap(K_1).$$ 
So overall
\[ \ehzcap(K_2) \leq \ehzcap(K) - \ehzcap(K_1), \]
but from sub-additivity we get the reverse inequality which implies equality.

\end{proof}

\medskip\medskip

Consider the FR index defined using the Alexander--Spanier cohomology as in \cite{fadellRabinowitz}.
Recall that for a metric space $X$ with an $S^1$ action, we define the $S^1$-fiber bundle 
$$\pi: X \times ES^1 \to X \times_{S^1} ES^1 := (X \times ES^1)/S^1,$$ 
where $ES^1 \to BS^1$ is the universal $S^1$ bundle over the classifying space.
We set 
$$\check{H}^*_{S^1}(X) := \check{H}^*(X \times_{S^1} ES^1),$$
where $\check{H}^*$ is the Alexander--Spanier cohomology groups.
We denote $e$ to be the pull-back by the projection $X \times_{S^1} ES^1 \to BS^1$ of the generator of $\check{H}^2(BS^1)$.
The FR index is defined to be
$$ \text{ind}_{\text{FR}}(X) = \sup\{ k+1 : e^k \neq 0\}.$$
A useful fact is that if the action is free, one has the Gysin long exact sequence 
$$ \cdots \xrightarrow{\pi^*} \check{H}^{*+1}(X) \xrightarrow{\pi_*} \check{H}^*_{S^1}(X) \xrightarrow{\cup_e} \check{H}^{*+2}_{S^1}(X) \xrightarrow{\pi^*} \check{H}^{*+2}(X)\to \cdots$$

Under our normalization, $\mathrm{Sys}(K)$ is naturally an $S^1$-space. 
Recall that $K$ is \emph{generalized Zoll} if the FR index
\[
\mathrm{ind}_{\mathrm{FR}}(\mathrm{Sys}(K)) \geq n.
\]

\medskip \medskip

    \begin{proof}[Proof of Theorem~\ref{additiveIsZollNonSmoothThm}]

Normalize $K$ to have $\ehzcap(K) = 1.$ 
Let $\gamma \in \text{Sys}(K).$

For $t \in S^1$ consider the loop
\[
 \eta_t := \gamma([0,t]) \# [\gamma(t),\gamma(0)],
\]
where $[\gamma(t),\gamma(0)]$ is the straight segment joining $\gamma(t)$ and $\gamma(0)$. A simple consequence of convexity of $K$ shows that the function
\[
g_\gamma(t) := \mathcal{A}(\eta_t)
\]
is monotone increasing in $t$.
Moreover, $C^0$ compactness of $\mathrm{Sys}(K) \times [0,1]$ implies that there exists a small $\delta>0$ such that the inverse value $g_\gamma^{-1}(\varepsilon)$ is well-defined and continuous both in $\gamma$ and in $\varepsilon$ for $\varepsilon$ in a neighborhood of $\delta$. Moreover, $\|\gamma(g_\gamma^{-1}(\varepsilon)) - \gamma(0)\| \ge c > 0$
for some uniform constant $c>0$ independent of $\gamma$.

Define a continuous map
\[
f_\varepsilon : \mathrm{Sys}(K) \to S^{2n-1}
\]
by
\[
f_\varepsilon(\gamma) = -J\frac{ \gamma\bigl(g_\gamma^{-1}(\varepsilon)\bigr) -\gamma(0)}{\bigl\|\gamma\bigl(g_\gamma^{-1}(\varepsilon)\bigr) - \gamma(0)\bigr\|}.
\]
Note that $f_\varepsilon(\gamma) = v$ if and only if a hyperplane cut with normal $v$ and with $\ehzcap(K_1) =\varepsilon$ cuts $\gamma$ additively.
By our assumptions, $f_\varepsilon$ is a surjection and every fiber $f_\varepsilon^{-1}(u)$ is contractible for every $u \in S^{2n-1}$.
The Vietoris-Begle mapping theorem together with compactness of $\text{Sys}(K)$ in the $C^0$ topology gives that  $\check{H}^k(\text{Sys}(K); C^0) = H^k(S^{2n-1})$ for every $k$.

When plugging this into the Gysin sequence above, we get that the FR index $\text{ind}_{\text{FR}}(\text{Sys}(K); C^0) \geq n.$ Following~\cite{Matijevic}, this proves that $K$ is generalized Zoll.
\end{proof}

\begin{remark}

We remark that
monotonicity of the FR index gives that the FR index of small $C^0$ $S^1$-equivariant open subsets $\mathcal{U} \subset W^{1,2}(S^1,\R^{2n})$ of $\text{Sys}(K)$ are also greater than or equal to $n$, and, being open subsets of a Banach space, their Alexander--Spanier cohomologies coincide with their singular cohomologies.

\end{remark}
\begin{remark}
     The fact that equivalence classes of $\sim$ are convex and hence contractible, together with Theorem~\ref{compactQuotientTopologyThm}, show that if $K$ does not admit isotropic gliding then
$$\check{H}^k(\text{Sys}(K); C^0) = \check{H}^k(\text{Sys}(K)/\sim; C^0) = \check{H}^k(\text{Sys}(K)/\sim; d_{L}).$$

\end{remark}

\section*{Appendix}

\begin{proof}[Proof of Proposition~\ref{simplexCapCubeIsZollThm}]
We work with the combinatorial description of cuts from
Lemma~\ref{combinatorialCutLemma}.

Recall that the outer unit normals of $Y$ are
$$
  \{\pm e_1,\pm e_2,\pm e_3,\pm e_4,e_s\},\qquad
  e_s := \Bigl(\tfrac12,\tfrac12,\tfrac12,\tfrac12\Bigr),
$$
and the support function satisfies
\[
  h_Y(-e_i)=0,\quad h_Y(e_i)=\tfrac12,\quad h_Y(e_s)=\tfrac12,\qquad i=1,2,3,4.
\]
The Lagrangian faces are spanned by one normal from
\(\{\pm e_1,\pm e_3\}\) and one from \(\{\pm e_2,\pm e_4\}\).

\medskip

\noindent\textbf{Step 1: Maximizing sequences for $Y$.}
We first classify maximizing sequences in \(M(Y)\) where each normal
appears at most once. 

Up to cyclic permutations and Lagrangian swaps there is a unique maximizing order in which each normal appears at
most once:
\[
  -e_3,\ e_1,\ -e_4,\ e_2,\ e_s,\ e_3,\ -e_1,\ e_4,\ -e_2.
\]
Indeed, one sees that any reordering that reverses the sign of some nonzero \(\omega\)-pair strictly decreases \(\sum_{i<j}\beta_i\beta_j\omega(n_i,n_j)\), independently of the actual values of the \(\beta_i\).

Now write the coefficients of the normals in a maximizing sequence as
follows:
\[
  a_i \text{ for } e_i,\quad
  b_i \text{ for } -e_i\ (i=1,\dots,4),\quad
  b_s \text{ for } e_s.
\]
The constraints defining \(M(Y)\) 
imply
\begin{equation}\label{branchConstraintsEq}
  b_i = a_i + \tfrac12 b_s,\qquad
  \sum_{i=1}^4 a_i + b_s = 2.
\end{equation}
A direct computation of the action term $A := \sum_{i<j}\beta_i\beta_j\omega(n_i,n_j)$
yields
\begin{equation}\label{AFormulaEq}
  A = 2a_1a_3 + 2a_2a_4 + 2b_s - \tfrac12 b_s^2.
\end{equation}
For fixed \(b_s\in[0,2]\), maximizing \(A\) under the constraint
\(\sum_{i=1}^4 a_i = 2-b_s\) gives exactly two branches:
\[
  \text{(I)}\quad
  a_1=a_3=\tfrac{2-b_s}{2},\ a_2=a_4=0,
  \qquad\text{or}\qquad
  \text{(II)}\quad
  a_1=a_3=0,\ a_2=a_4=\tfrac{2-b_s}{2}.
\]
In branch~(I), using \eqref{branchConstraintsEq},$b_1=b_3=1,\quad b_2=b_4=\tfrac{b_s}{2},$
and in branch~(II) the roles of \((1,3)\) and \((2,4)\) are exchanged and one has $b_2=b_4=1,\quad b_1=b_3=\tfrac{b_s}{2}.$
Substituting into \eqref{AFormulaEq} gives, in both cases,
\[
  A
  = \frac{(2-b_s)^2}{2} + 2b_s - \frac{b_s^2}{2}
  = 2,
\]
so both branches consist of maximizing sequences for all
\(b_s\in[0,2]\). At \(b_s=2\) the two branches meet at $a_1=a_2=a_3=a_4=0.$

We now allow sequences in which some normals appear more than once.
A \emph{cross} is a pattern in which a normal appears twice, with a block of normals in-between where inside this block it has \(\omega\)-pairings of both signs
whose sum is zero; such a pattern cannot be eliminated by Lagrangian swaps
and cyclic permutations alone.
Consider branch~(I) with \(b_s>0\), so
\(a_2=a_4=0\) and all other coefficients are strictly positive.
Then the only possible cross is when the normal \(e_s\) “passes over’’ both
\(e_2\) and \(e_4\). This yields the
additional maximizing sequence
\[
  -e_3,\ e_1,\ e_s,\ -e_2,\ -e_4,\ e_s,\ e_3,\ -e_1.
\]
A symmetric cross arises in branch~(II),
\[
  -e_4,\ e_2,\ e_s,\ -e_1,\ -e_3,\ e_s,\ e_4,\ -e_2.
\]
One easily checks that each such cross can occur at most once.

\medskip

\noindent\textbf{Step 2: Combinatorial cuts for a given direction.}\\
Fix a direction $v = \sum_{i=1}^4 \alpha_i e_i,$
and consider hyperplane cuts with normal \(v\).
By Lemma~\ref{combinatorialCutLemma} and
Remark~\ref{arbitraryCutRemark}, it suffices to show existence of families of maximizing sequences in \(M(Y)\) of the following form:

\begin{itemize}
\item For each sequence in the family, there is a block \((\beta_i,n_i)_{i=1}^m\) consisting of unit
  outer normals for one side of the cut (say, for \(Y_1\)) such that
  \(\sum_{i=1}^m \beta_i n_i = c\,v\) with all \(\beta_i\ge 0\) and
  \(c>0\).
\item The remaining block \((\beta_i,n_i)_{i=m+1}^k\) completes this to a global
  maximizing sequence in \(M(Y)\) in one of the branches described
  in Step~1.
\item When $v \in \text{int}(N_K(x))$ for some vertex $x \in \partial Y$, 
    there is such a sequence for every $c>0$ sufficiently small.
Note that the inequality in
  Remark~\ref{arbitraryCutRemark} holds for all sufficiently small \(c\).
  Hence for each small \(c>0\) we obtain a combinatorial cut for some
  \(t=t(c)\in(0,h_Y(v))\).
\item For each $c$, the space of such combinatorial cuts is contractible.

\end{itemize}

We now go over three ``representative" cases, and we leave the rest as an exercise.

\noindent\textbf{Case (I):} \(\alpha_3,\alpha_4 > \alpha_1,\alpha_2\) and
\(\alpha_3,\alpha_4>0\).

Then
\[
  v \in N_Y\bigl(0,0,\tfrac12,\tfrac12\bigr),
\]
whose extreme normals are
\(\{e_3,e_4,-e_1,-e_2,e_s\}\). Assume without loss of generality
\(\alpha_3>\alpha_4\). This restricts to branch~(I) where \(a_2=a_4=0\), so \(e_4\)
cannot appear in the first block of the cut.

Writing
\[
  cv = 2c\alpha_4 e_s
        + c(\alpha_3-\alpha_4)e_3
        - c(\alpha_4-\alpha_1)e_1
        - c(\alpha_4-\alpha_2)e_2,
\]
we see this is the unique non-negative linear combination of $cv$.
We take the first block $(n_1,\dots,n_m)$ to be the four normals $e_s,\ e_3,\ -e_1,\ -e_2$
in that order; in this order all nonzero $\omega$-pairings within
the block are positive, which implies that this order uniquely maximizes the block.

In branch~(I) the total coefficients must satisfy
\[
  b_1=b_3=1,\qquad
  b_2=b_4=\tfrac{b_s}{2},
\]
and $a_2=a_4=0$. Note that $e_3$ cannot appear in the second block. Indeed, its coefficient in the first block is positive, and it lies between $e_s$ and $-e_1$ and cannot be connected using Lagrangian swaps to an instance of $e_3$ in the second block.
Hence its coefficient must match
the global $a_3$, we must have
\[
  a_3 = c(\alpha_3-\alpha_4),\qquad
  a_1 = c(\alpha_3-\alpha_4),
\]
and consistency with \eqref{branchConstraintsEq} forces
\[
  b_s = 2 - 2c(\alpha_3-\alpha_4),\quad
  b_1=b_3=1,\quad
  b_2=b_4 = 1 - c(\alpha_3-\alpha_4).
\]
With these coefficients, a combinatorial cut (up to
Lagrangian swaps) is
\[
\begin{aligned}
(n_1,\dots,n_m)
  &= \bigl(
      2c\alpha_4\, e_s,\,
      c(\alpha_3-\alpha_4)\,e_3,\,
      c(\alpha_4-\alpha_1)\,(-e_1),\,
      c(\alpha_4-\alpha_2)\,(-e_2)
     \bigr),\\
(n_{m+1},\dots,n_k)
  &= \bigl(
      (1-c(\alpha_4-\alpha_1))\,(-e_1),\,
      (1-c(\alpha_3-\alpha_2))\,(-e_2),\,
      -e_3,\\
  &\qquad
      c(\alpha_3-\alpha_4)\,e_1,\,
      (1-c(\alpha_3-\alpha_4))\,(-e_4),\,
      (2-2c\alpha_3)\,e_s
     \bigr).
\end{aligned}
\]

In the first block the only possible repetitions of normals arise from
splitting coefficients along Lagrangian faces (e.g.\ moving some portion of
$-e_2$ across $-e_1$ and then over $e_3$). These moves form a convex
parameter set and correspond exactly to Lagrangian swaps; hence all such
configurations deformation retract to the canonical one in which each
normal appears exactly once.

In the second block there is one additional family of configurations:
$e_s$ may “cross over’’ $-e_4$ and $-e_2$. After suitable Lagrangian swaps
we can arrange that the coefficients of the two occurrences of $-e_2$ and
$-e_4$ adjacent to $e_s$ are equal, so that $\omega(-a e_2 - a e_4,\,e_s) = 0,$
and this crossing is controlled by a single real parameter $a$. This family
is contractible and admits a canonical deformation back to the configuration
above, where each normal appears once. Thus, for Case~(I) and for all
sufficiently small $c>0$, the space of combinatorial cuts is contractible.

\medskip

\noindent\textbf{Case (II):} Exactly one $\alpha_i$ is positive.
Without loss of generality, $\alpha_3>0,$ $\alpha_1,\alpha_2,\alpha_4<0.$
Then $v \in N_Y\bigl(0,0,\tfrac12,0\bigr),$
whose extreme normals are \(\{-e_2,-e_4,e_3,-e_1\}\). In this normal cone $e_s$
does not appear, so we must write
\[
  cv
   = -c\alpha_2(-e_2)
     - c\alpha_4(-e_4)
     + c\alpha_3 e_3
     - c\alpha_1(-e_1),
\]
with all coefficients nonnegative (for $c>0$). Up to Lagrangian swaps this
gives a unique order for the first block.

In this case $e_s$ can only appear in the second block, and the only
restriction on $b_s$ is that all coefficients remain nonnegative, which
yields
\[
  2c\max\{-\alpha_2,-\alpha_4\} \;\le\; b_s \;\le\; 2-2c\alpha_3.
\]
An explicit combinatorial cut is then
\[
\begin{aligned}
(n_1,\dots,n_m)
  &= \bigl(
      -c\alpha_2\,(-e_2),\,
      -c\alpha_4\,(-e_4),\,
      c\alpha_3\,e_3,\,
      -c\alpha_1\,(-e_1)
     \bigr),\\
(n_{m+1},\dots,n_k)
  &= \bigl(
      (1+c\alpha_1)\,(-e_1),\,
      -e_3,\,
      \tfrac{2-b_s}{2}\,e_1,\,
      \bigl(\tfrac{b_s}{2}+c\alpha_4\bigr)\,(-e_4),\\
  &\qquad
      b_s\,e_s,\,
      \bigl(\tfrac{2-b_s}{2}-c\alpha_3\bigr)\,e_3,\,
      \bigl(\tfrac{b_s}{2}+c\alpha_2\bigr)\,(-e_2)
     \bigr).
\end{aligned}
\]
Here all coefficients are nonnegative for $c>0$ small and
\(b_s\) in the interval above. There are no nontrivial cross patterns:
each normal either appears only once in each block or splits along
Lagrangian faces in a convex way, and such splittings are removed by
Lagrangian swaps. Hence for fixed $c>0$ the space of cuts is an interval in the parameter $b_s$, and is therefore contractible.

\medskip

\noindent\textbf{Case (III):} \(\alpha_1,\alpha_2,\alpha_3,\alpha_4 < 0\).

In this case $v$ belongs to the normal cone of the vertex
\(\bigl(0,0,0,0\bigr)\), and
$cv$ is a nonnegative combination of \(-e_1,-e_2,-e_3,-e_4\). 
One obtains a unique linear combination for $cv$, and the resulting maximizing sequences for $Y$
described in Step~1 contribute two branches of solutions for $b_s$: Either $a_1=a_3=0$, and the admissible values of $b_s$ are $b_s \in \bigl[\,2c\max\{-\alpha_1,-\alpha_3\},\,2\,\bigr]$, or $a_2=a_4=0$ and $b_s \in \bigl[\,2c\max\{-\alpha_2,-\alpha_4\},\,2\,\bigr].$
At $b_s=2$ the two branches meet, and in both branches the coefficients of
all normals are linear functions of $c$ and $b_s$. As in the previous
cases, the only possible degeneracies come from splitting coefficients
along Lagrangian faces, which form convex parameter sets and can be
straightened by Lagrangian swaps. Therefore, for each fixed $c>0$ small
enough, the union of the two branches is an interval in $b_s$ with the two
ends glued at $b_s=2$, hence contractible. 

\medskip

\noindent\textbf{Step 3: Applying Theorem~\ref{additiveIsZollNonSmoothThm}} \\
In each of the cases above, for every direction $v$ we have constructed,
for all sufficiently small $c>0$, families of maximizing sequences
$(\beta_i(c),n_i)$ with $\sum_{i=1}^m \beta_i(c)\,n_i = c\,v,$
such that the coefficients in the first block are linear in $c$ and
strictly positive for small $c$. Thus $\sum_{1\le i<j\le m}\beta_i(c)\beta_j(c)\omega(n_i,n_j)
\sim O(c^2)$
while $\sum_{i=1}^m \beta_i(c) h_Y(n_i)
\sim O(c),$
so for sufficiently small $c>0$ the inequality in
Remark~\ref{arbitraryCutRemark} holds:
\[
  0
  < \sum_{1\le i<j\le m}\beta_i(c)\beta_j(c)\omega(n_i,n_j)
  < \frac{c\,h_Y(v)}{2\,\ehzcap(Y)}.
\]
By Remark~\ref{arbitraryCutRemark}, for each such $c$ there exists a
(unique) $t=t(c)\in(0,h_Y(v))$ such that $(\beta_i(c),n_i)$ is a
combinatorial cut for $(v,t(c))$. As $c\to 0$, we have $t(c)\to 0$.
Moreover, for each fixed $c$ the space of combinatorial cuts is contractible, so the fibers over $t(c)$
are contractible as well.

When $v$ lies in the intersection of normal cones of several vertices of $Y$,
the family of admissible parameters $c$ produced by the constructions above
may fail to extend all the way down to arbitrarily small values of $c$.
In those degenerate directions one can no longer appeal directly to
Remark~\ref{arbitraryCutRemark} to obtain small cuts. Instead, for such $v$
one verifies directly, using Definition~\ref{combinatorialCutDef} and the
same type of combinatorial analysis of coefficients and orders of normals
as in Cases~(I)–(III), that for all sufficiently small $t>0$ there exist
combinatorial cuts for $(v,t)$ and that the corresponding spaces of
combinatorial cuts are still contractible.

Thus, for every direction $v$ there exists $\varepsilon(v)>0$ such that
for all $t\in(0,\varepsilon(v))$ the cut of $Y$ by the hyperplane
$\{x\mid \langle x,v\rangle = h_Y(v)-t\}$ is additive and the space of
combinatorial cuts is contractible. Lemma~\ref{combinatorialCutLemma}
then implies that the corresponding space $f_\varepsilon^{-1}(v)$ of
action-minimizing closed characteristics cut by that hyperplane is
contractible. By Theorem~\ref{additiveIsZollNonSmoothThm}, this shows that
$Y$ is generalized Zoll.

\end{proof}

\medskip \medskip \medskip

\begin{proof}[Proof of Proposition~\ref{pentagonProp}]

Similarly to the proof of Proposition~\ref{simplexCapCubeIsZollThm}, we show how to employ Lemma~\ref{combinatorialCutLemma} in certain situations, and the rest follow in exactly the same way.

Let $v_1,\ldots,v_5$ be the vertices of $P$ and $w_1,\ldots,w_5$ the vertices of $T$.
Let $n_1,\ldots,n_5$ be the unit outer normals to $P$ so that $n_i$ is orthogonal to $[v_{i-1},v_i]$.
Similarly, $\widetilde{n}_i$ is the unit outer normal to $T$ orthogonal to the segment $[w_{i-1},w_i].$

\medskip

Assume first that $u \in N_K((v_i, w_i)),$ i.e. $u = \alpha_1 n_i + \alpha_2 n_{i+1} + \alpha_3 \widetilde{n}_i + \alpha_4 \widetilde{n}_{i+1}.$
Denote $a, h, d$ to be the edge length, height, and the diagonal length of the pentagon respectively.
A direct computation shows that the sequence
\[
\frac{1}{(4a + 2d)h}\bigl(
      a\,\widetilde{n}_{i+1},\,
      x\,n_{i+1},\,
      (a-x)\,n_i,\, 
      a\,\widetilde{n}_{i},\,
      (a+x(\tfrac{d}{a}-1))\,n_{i+4},\,
      d\,\widetilde{n}_{i+3},\,
      (d-x(\tfrac{d}{a}-1))\,n_{i+2}\,
     \bigr)
\]
is a maximizer for \eqref{formulaEq} in $M(P \times T)$ for every $x \in [0, a]$.
Consider the sequences
\[
\begin{aligned}
  &\frac{1}{(4a + 2d)h}\bigl(
      c\alpha_4\,\widetilde{n}_{i+1},\,
      c\alpha_2\,n_{i+1},\,
      c\alpha_1\,n_{i},\,
      c\alpha_3\,\widetilde{n}_i
     \bigr),\\
  &\frac{1}{(4a + 2d)h}\bigl(
      (a-c (\alpha_1+\alpha_2))\,n_i,\,
      (a-c \alpha_3)\,\widetilde{n}_i,\,
      (a+c\alpha_2(\tfrac{d}{a}-1))\,n_{i+4},\,\\
      &\qquad\qquad\qquad d\,\widetilde{n}_{i+3},\,
      (d-c\alpha_2(\tfrac{d}{a}-1))\,n_{i+2},\,
      (a-c \alpha_4)\,\widetilde{n}_{i+1}
     \bigr).
\end{aligned}
\]
Combining these sequences together to a single sequence, up to Lagrangian swaps, gives the maximizing sequence above with $x=c \alpha_2.$ 
If the action term $A_1$ of the first sequence is positive, these sequences form a combinatorial cut for every small $c$, and invoking Lemma~\ref{combinatorialCutLemma} we get additivity for $(v,t)$ with small enough $t$.

The action term of the first sequence is $A_1 = C (\alpha_1 \alpha_4 - \alpha_2 \alpha_3) $ for some constant $C > 0.$
Hence for the case $\alpha_1 \alpha_4 > \alpha_2 \alpha_3$ this defines a combinatorial cut, and for the analogous case $\alpha_2 \alpha_3 > \alpha_1 \alpha_4$ one is able to define similar sequences.

\medskip

Let us conclude with the case $\alpha_3=\alpha_4=0,$ and $\alpha_1 < \alpha_2.$
Consider the maximizer in $M(P\times T)$ of the form
\[
\frac{1}{(4a + 2d)h}\bigl(
      x\,n_i,\, 
      (d-x(\tfrac{d}{a}-1))\,n_{i+1},\,
      d\,\widetilde{n}_{i+4},\,
      (a-x)\,n_{i+2},\,
      (a+x(\tfrac{d}{a}-1))\,n_{i+3},\,
      a\,\widetilde{n}_{i+2},\,
      a\,\widetilde{n}_{i+1}\,
     \bigr)
\]
One can choose $x$ so that $\frac{x}{d-x(\tfrac{d}{a}-1)} = \frac{\alpha_1}{\alpha_2},$ and then get the following combinatorial cut for every small $\varepsilon >0.$
\[
\begin{aligned}
  &\frac{1}{(4a + 2d)h}\bigl(
      \varepsilon a\,\widetilde{n}_{i+2},\,
      \varepsilon a\,\widetilde{n}_{i+1},\,
      x\,n_i,\, 
      (d-x(\tfrac{d}{a}-1))\,n_{i+1},\,
      \varepsilon d\,\widetilde{n}_{i+3}\,
     \bigr),\\
  &\frac{1}{(4a + 2d)h}\bigl(
      (1-\varepsilon)d\,\widetilde{n}_{i+3},\,
      (a-x)\,n_{i+3},\,
      (a+x(\tfrac{d}{a}-1))\,n_{i+2},\,
      (1-\varepsilon)a\,\widetilde{n}_{i+2},\,
      (1-\varepsilon)a\,\widetilde{n}_{i+1}\,
     \bigr).
\end{aligned}
\]
In this case the coefficient of the direction $v$ is not small. However, one can verify using Lemma~\ref{combinatorialCutLemma} that this sequences indeed give a combinatorial cut for $(v,t)$ for every small enough $t$.

\end{proof}

\bibliography{references1}

\begin{thebibliography}{AAMO08}

\bibitem[AAC24]{AAC}
H. Alizadeh, M.~S. Atallah, and D. Cant.
\newblock The spectral diameter of a symplectic ellipsoid.
\newblock {\em arXiv:2408.07214}, 2024.

\bibitem[AAKO14]{capacity_mahler}
S. Artstein-Avidan, R. Karasev, and Y. Ostrover.
\newblock From symplectic measurements to the {M}ahler conjecture.
\newblock {\em Duke Math. J.}, 163(11):2003–2022, 2014.

\bibitem[AAMO08]{ArtsteinAvidanOstroverMilman}
S. Artstein-Avidan, V. Milman, and Y. Ostrover.
\newblock The {$M$}-ellipsoid, symplectic capacities and volume.
\newblock {\em Comment. Math. Helv.}, 83(2):359--369, 2008.

\bibitem[AAO08]{ArtsteinAvidanOstrover2008}
S. Artstein-Avidan and Y. Ostrover.
\newblock A {B}runn-{M}inkowski inequality for symplectic capacities of convex domains.
\newblock {\em Int. Math. Res. Not. IMRN}, (13):Art. ID rnn044, 31, 2008.

\bibitem[AAO14]{AA-O}
S. Artstein-Avidan and Y. Ostrover.
\newblock Bounds for {M}inkowski billiard trajectories in convex bodies.
\newblock {\em Int. Math. Res. Not. IMRN}, (1):165–193, 2014.

\bibitem[AB23]{abbondandoloBenedetti}
A. Abbondandolo and G. Benedetti.
\newblock On the local systolic optimality of {Z}oll contact forms.
\newblock {\em Geom. Funct. Anal.}, 33(2):299--363, 2023.

\bibitem[ABE25]{abbondandolo-benedetti-edtmair}
A. Abbondandolo, G. Benedetti, and O. Edtmair.
\newblock Symplectic capacities of domains close to the ball and {B}anach-{M}azur geodesics in the space of contact forms.
\newblock {\em Duke Math. J.}, 174(8):1567--1646, 2025.

\bibitem[ABHS18]{ABHS}
A. Abbondandolo, B. Bramham, U.~L. Hryniewicz, and P.~A.~S. Salom{\~{a}}o.
\newblock Sharp systolic inequalities for {R}eeb flows on the three-sphere.
\newblock {\em Invent. Math.}, 211(2):687--778, 2018.

\bibitem[ABKS16]{non_symmetric_mahler}
A. Akopyan, A. Balitskiy, R. Karasev, and A. Sharipova.
\newblock Elementary approach to closed billiard trajectories in asymmetric normed spaces.
\newblock {\em Proc. Amer. Math. Soc.}, 144(10):4501–4513, 2016.

\bibitem[AK22]{abbondandolo-kang}
A. Abbondandolo and J. Kang.
\newblock Symplectic homology of convex domains and {C}larke's duality.
\newblock {\em Duke Math. J.}, 171(3):739--830, 2022.

\bibitem[Bal20]{balitsky}
A. Balitskiy.
\newblock Equality cases in {V}iterbo's conjecture and isoperimetric billiard inequalities.
\newblock {\em Int. Math. Res. Not. IMRN}, (7):1957--1978, 2020.

\bibitem[BB09]{bezdek-bezdek}
D. Bezdek and K. Bezdek.
\newblock Shortest billiard trajectories.
\newblock {\em Geom. Dedicata}, 141:197--206, 2009.

\bibitem[BH12]{bramhamHofer}
B. Bramham and H. Hofer.
\newblock First steps towards a symplectic dynamics.
\newblock In {\em Surveys in differential geometry. {V}ol. {XVII}}, volume~17 of {\em Surv. Differ. Geom.}, pages 127--177. Int. Press, Boston, MA, 2012.

\bibitem[BW58]{boothbyWang}
W.~M. Boothby and H.~C. Wang.
\newblock On contact manifolds.
\newblock {\em Ann. of Math. (2)}, 68:721--734, 1958.

\bibitem[Can25]{cant}
D. Cant.
\newblock Hamiltonian linking and symplectic packing.
\newblock {\em arXiv:2507.01416}, 2025.

\bibitem[CE22]{dynconvnotconvdim4}
J. Chaidez and O. Edtmair.
\newblock 3{D} convex contact forms and the {R}uelle invariant.
\newblock {\em Invent. Math.}, 229(1):243--301, 2022.

\bibitem[CH21]{Ch-H}
J. Chaidez and M. Hutchings.
\newblock Computing {R}eeb dynamics on four-dimensional convex polytopes.
\newblock {\em J. Comput. Dyn.}, 8(4):403--445, 2021.

\bibitem[CHLS07]{capacity_survey_1}
K. Cieliebak, H. Hofer, J. Latschev, and F. Schlenk.
\newblock Quantitative symplectic geometry.
\newblock In {\em Dynamics, ergodic theory, and geometry}, volume~54 of {\em Math. Sci. Res. Inst. Publ.}, pages 1--44. Cambridge Univ. Press, Cambridge, 2007.

\bibitem[Cla79]{clarke}
F.~H. Clarke.
\newblock A classical variational principle for periodic {H}amiltonian trajectories.
\newblock {\em Proc. Amer. Math. Soc.}, 76(1):186--188, 1979.

\bibitem[Edt24]{edtmair}
O. Edtmair.
\newblock Disk-like surfaces of section and symplectic capacities.
\newblock {\em Geom. Funct. Anal.}, 34(5):1399--1459, 2024.

\bibitem[EH89]{hofer-ekeland}
I. Ekeland and H. Hofer.
\newblock Symplectic topology and {H}amiltonian dynamics.
\newblock {\em Math. Z.}, 200(3):355--378, 1989.

\bibitem[FR78]{fadellRabinowitz}
E.~R. Fadell and P.~H. Rabinowitz.
\newblock Generalized cohomological index theories for {L}ie group actions with an application to bifurcation questions for {H}amiltonian systems.
\newblock {\em Invent. Math.}, 45(2):139--174, 1978.

\bibitem[FS07]{frauenfelderSchlenk}
U. Frauenfelder and F. Schlenk.
\newblock Hamiltonian dynamics on convex symplectic manifolds.
\newblock {\em Isr. J. Math.}, 159:1--56, 2007.

\bibitem[Gei08]{Geiges_2008}
H. Geiges.
\newblock {\em An Introduction to Contact Topology}.
\newblock Cambridge Studies in Advanced Mathematics. Cambridge University Press, 2008.

\bibitem[GGM21]{Gi-Gu-Ma}
V.~L. Ginzburg, B.~Z. G\"{u}rel, and M. Mazzucchelli.
\newblock On the spectral characterization of {B}esse and {Z}oll {R}eeb flows.
\newblock {\em Ann. Inst. H. Poincar\'{e} C Anal. Non Lin\'{e}aire}, 38(3):549--576, 2021.

\bibitem[GH18]{Gu-Ha}
J. Gutt and M. Hutchings.
\newblock Symplectic capacities from positive $\mathrm{S}^1$-equivariant symplectic homology.
\newblock {\em Algebr. Geom. Topol.}, 18:3537--3600, 2018.

\bibitem[GHR22]{gutt-hutchings-ramos}
J. Gutt, M. Hutchings, and V.~G.~B. Ramos.
\newblock Examples around the strong {V}iterbo conjecture.
\newblock {\em J. Fixed Point Theory Appl.}, 24(2):Paper No. 41, 22, 2022.

\bibitem[GR24]{GuttRamos}
J. Gutt and V.~G. Ramos.
\newblock The equivalence of {E}keland-{H}ofer and equivariant symplectic homology capacities.
\newblock {\em arXiv:2412.09555}, 2024.

\bibitem[Gro85]{gromov}
M. Gromov.
\newblock Pseudo holomorphic curves in symplectic manifolds.
\newblock {\em Invent. Math.}, 82(2):307--347, 1985.

\bibitem[Her98]{hermann}
D. Hermann.
\newblock Non-equivalence of symplectic capacities for open sets with restricted contact type boundary.
\newblock {\em Prépublication d’Orsay numéro 32}, 1998.

\bibitem[HK19a]{pazit-thesis-msc}
P. Haim-Kislev.
\newblock The {EHZ} capacity of cubes and simplices.
\newblock {\em MSc Thesis at Tel-Aviv University}, 2019.

\bibitem[HK19b]{pazit}
P. Haim-Kislev.
\newblock On the symplectic size of convex polytopes.
\newblock {\em Geom. Funct. Anal.}, 29(2):440--463, 2019.

\bibitem[HKO25]{counterexample}
P. Haim-Kislev and Y. Ostrover.
\newblock A counterexample to viterbo's conjecture.
\newblock {\em To appear in Ann. of Math. (2)}, 2025.

\bibitem[HWZ98]{HWZ-3dconvex}
H. Hofer, K. Wysocki, and E. Zehnder.
\newblock The dynamics on three-dimensional strictly convex energy surfaces.
\newblock {\em Ann. of Math. (2)}, 148(1):197--289, 1998.

\bibitem[HZ94]{hofer_zehnder}
H. Hofer and E. Zehnder.
\newblock {\em Symplectic Invariants and {H}amiltonian Dynamics}.
\newblock Birkhauser Advanced Texts, Birkhauser Verlag, 1994.

\bibitem[Iri22]{irie}
K. Irie.
\newblock Symplectic homology of fiberwise convex sets and homology of loop spaces.
\newblock {\em J. Symplectic Geom.}, 20(2):417--470, 2022.

\bibitem[KS19]{karasev-sharipova}
R. Karasev and A. Sharipova.
\newblock Viterbo's conjecture for certain {H}amiltonians in classical mechanics.
\newblock {\em Arnold Math. J.}, 5(4):483--500, 2019.

\bibitem[K{\"u}n90]{singular_survey_1}
A. K{\"u}nzle.
\newblock Une capacit\'e symplectique pour les ensembles convexes et quelques applications.
\newblock {\em Ph.D. thesis, Universit\'e Paris IX Dauphine}, 1990.

\bibitem[K{\"u}n96]{Kunzle}
A. K{\"u}nzle.
\newblock Singular {H}amiltonian systems and symplectic capacities.
\newblock {\em Singularities and differential equations, Banach Center Publ., 33, Polish Acad. Sci., Warsaw}, pages 171--187, 1996.

\bibitem[Mat25]{Matijevic}
S. Matijevi\'{c}.
\newblock {S}ystolic ${S}^1$-index and characterization of non-smooth {Z}oll convex bodies.
\newblock {\em arXiv:2501.13856}, 2025.

\bibitem[McD14]{capacity_survey_2}
D. McDuff.
\newblock Symplectic topology today.
\newblock {\em AMS Joint Mathematics Meeting}, 2014.

\bibitem[Oh05a]{Oh2005}
Y.-G. Oh.
\newblock Construction of spectral invariants of {H}amiltonian paths on closed symplectic manifolds.
\newblock In {\em The breadth of symplectic and {P}oisson geometry}, volume 232 of {\em Progr. Math.}, pages 525--570. Birkh\"auser Boston, Boston, MA, 2005.

\bibitem[Oh05b]{oh-spectral}
Y.-G. Oh.
\newblock {Spectral invariants, analysis of the Floer moduli space, and geometry of the Hamiltonian diffeomorphism group}.
\newblock {\em Duke Mathematical Journal}, 130(2):199 -- 295, 2005.

\bibitem[Ost14]{capacity_survey_3}
Y. Ostrover.
\newblock When symplectic topology meets {B}anach space geometry.
\newblock {\em proceedings of the ICM}, 2:959--981, 2014.

\bibitem[Rab78]{rabinowitz78}
P.~H. Rabinowitz.
\newblock Periodic solutions of {H}amiltonian systems.
\newblock {\em Comm. Pure Appl. Math.}, 31(2):157--184, 1978.

\bibitem[Rud22a]{Mink-Bill-Rudolf}
D. Rudolf.
\newblock The {M}inkowski billiard characterization of the {EHZ}-capacity of convex {L}agrangian products.
\newblock {\em Journal of Dynamics and Differential Equations}, 2022.

\bibitem[Rud22b]{rudolf}
D. Rudolf.
\newblock Viterbo's conjecture for {L}agrangian products in {$\mathbb{R}^4$} and symplectomorphisms to the {E}uclidean ball.
\newblock {\em arXiv:2203.02294}, 2022.

\bibitem[Sch93]{subdiff_1}
R. Schneider.
\newblock {\em Convex Bodies: the {B}runn-{M}inkowski theory}.
\newblock Encyclopedia of Mathematics and its Applications, 44. Cambridge University Press, 1993.

\bibitem[Sch00]{Schwarz2000}
M. Schwarz.
\newblock On the action spectrum for closed symplectically aspherical manifolds.
\newblock {\em Pacific J. Math.}, 193(2):419--461, 2000.

\bibitem[Vit87]{viterboWeinsteinConj}
C. Viterbo.
\newblock A proof of {W}einstein's conjecture in {${\bf R}^{2n}$}.
\newblock {\em Ann. Inst. H. Poincar\'e{} Anal. Non Lin\'eaire}, 4(4):337--356, 1987.

\bibitem[Vit92]{Viterbo-specGF}
C. Viterbo.
\newblock Symplectic topology as the geometry of generating functions.
\newblock {\em Math. Ann.}, 292(4):685--710, 1992.

\bibitem[Vit00]{viterbo2000}
C. Viterbo.
\newblock Metric and isoperimetric problems in symplectic geometry.
\newblock {\em J. Amer. Math. Soc.}, 13(2):411--431, 2000.

\bibitem[Wei78]{weinstein78}
A. Weinstein.
\newblock Periodic orbits for convex {H}amiltonian systems.
\newblock {\em Ann. of Math. (2)}, 108(3):507--518, 1978.

\bibitem[Zeh13]{ball_add}
K. Zehmisch.
\newblock The codisc radius capacity.
\newblock {\em Electron. Res. Announc. Math. Sci.}, 20:77--96, 2013.

\end{thebibliography}
\bibliographystyle{my_alpha}

\end{document}